\documentclass[12pt]{article}
\usepackage{amsmath,amssymb,amsfonts}    
\usepackage{amsthm}
\usepackage{dsfont}
\usepackage{mathrsfs}

\usepackage[hyperref, usenames, svgnames,x11names, dvipsnames]{xcolor}
\usepackage[ocgcolorlinks, colorlinks=true, citecolor=Cerulean, linkcolor=Bittersweet, urlcolor=blue]{hyperref}
\newtheorem{thm}{Theorem}

\newtheorem{cor*}{Corollary}
\newtheorem{lem}{Lemma}
\newtheorem{prop}{Proposition}
\newtheorem{proposition}{Proposition}
\newtheorem{remark}{Remark}
\newtheorem{ex}{Example}[section]

\newtheorem{definition}{Definition}
\newtheorem{assumption}{Assumption}

\newtheorem{example}[thm]{Example}

\def\L{\mathscr{L}}

\def\dif{\:\mathrm{d}}

\def\0{\emptyset}

\def\dif{\:\mathrm{d}}
\def\L{\mathscr{L}}

\def\dif{\:\mathrm{d}}
\def\L{\mathscr{L}}

\title{Stability of the solution set of quasi-variational inequalities and optimal control \footnotetext{{This research was carried out in the framework of {\sc MATHEON} supported by the Einstein Foundation Berlin within the
ECMath projects OT1, SE5 and SE19 as well as the Excellence Cluster Math$^+$ Berlin within project AA4-3, EF3-5, the transition project OT6, and the {\sc MATHEON} project A-AP24. The authors further acknowledge the
support of the DFG through the DFG-SPP 1962: Priority Programme Non-smooth and Complementarity-based
Distributed Parameter Systems: Simulation and Hierarchical Optimization  within Projects 10 and 11. CNR was supported by the NSF grant DMS-2012391. The authors further thank the two anonymous reviewers that helped simplify and improve the results.}}}

\author{Amal Alphonse\thanks{Weierstrass Institute, Mohrenstrasse 39, 10117 Berlin, Germany ({\tt alphonse@wias-berlin.de},  {\tt hintermueller@wias-berlin.de})} \and Michael Hinterm\"{u}ller$^\dagger$\thanks{Department of Mathematics, Humboldt-University of Berlin, Unter den Linden 6, 10099 Berlin, Germany. ({\tt hint@math.hu-berlin.de}).} 
        \and Carlos N. Rautenberg\thanks{Department of Mathematical Sciences and the Center for Mathematics and Artificial Intelligence (CMAI), George Mason University, Fairfax, VA 22030, USA. ({\tt crautenb@gmu.edu})
}
    }

\begin{document}

\maketitle
\begin{abstract}
For a class of quasi-variational inequalities (QVIs) of obstacle-type the stability of its solution set and associated optimal control problems are considered. These optimal control problems are non-standard in the sense that they involve an objective with set-valued arguments. The approach to study the solution stability is based on perturbations of minimal and maximal elements of the solution set of the QVI with respect to {monotone} perturbations of the forcing term. It is shown that different assumptions are required for studying decreasing and increasing perturbations and  that the optimization problem of interest is well-posed.

\end{abstract}

%

\pagestyle{myheadings}
\thispagestyle{plain}

\section{Introduction}\label{section:intro}

Quasi-variational inequalities (QVIs) are powerful mathematical models for the description of complex physical phenomena. Such models arise in many scientific areas including superconductivity \cite{Rodrigues2000,MR1765540,MR2947539, MR2652615,MR3335194,Prigozhin,MR3648950,MR3119319,MR3023771}, continuum mechanics \cite{Friedman1982}, impulse control problems \cite{Bensoussan1974,MR673169,Bensoussan1982,Bensoussan1984}, growth of sandpiles \cite{MR3082292,MR3231973,MR3335194,Prigozhin1986,Prigozhin1994,Prigozhin1996,Prigozhin2012}, and the formation of networks of lakes and rivers \cite{MR3231973,Prigozhin1994,Prigozhin1996}, among others. 

In general, QVIs are nonlinear, nonconvex, and nonsmooth problems with non-unique (i.e., set-valued) solutions. In physical models like the growth of sandpiles or the determination of the magnetic field in superconductors, each of these solutions fulfills physical laws confirming that they are not artifacts of the mathematical formulation (compare the results in \cite{MR2652615,MR3082292,MR3231973,MR3335194,Prigozhin1986,MR3800749}). In some cases, like the QVI arising in impulse control problems, extremals of the solution set can be determined, in the sense that there exist minimal and maximal elements of the solution set which are related to the value functional \cite{Bensoussan1974}.

The mathematical treatment of QVIs entails several possible directions. In addition to the ``order'' approach followed in this paper, at least two more are worth mentioning. In some cases, the QVI can be expressed as a generalized equation, and hence a particular instance of a more general problem class; see, e.g., \cite{MR2571488,MR2997552} and also \cite{MR3238849,MR3307334,MR3299010}. In problems involving constraints on derivatives, special forms of constraint regularization that modify the original partial differential operator may be suitable, see  \cite{MR1765540, MR2861829,MR2979609,MR3213477}. For details on these and further approaches, we refer the reader to \cite{alphonse2019recent}.

Given the complexity of QVIs, their optimal control represents a task which is yet even more complex than the study of the QVI itself. Without any structural properties of the solution set, the treatment of the control problem appears very hard, if not impossible. However, properties such as the availability of extremal elements
provide useful starting points for the successful analysis of the control problem and characterizations of its solutions. For this purpose, the study of the stability of minimal and maximal elements of the solution set with respect to perturbations of the forcing term represents a fundamental analytical step for the subsequent study of the control problem. Concerning the latter considered in infinite dimensions, we note that the literature is rather scarce; we refer to \cite{MR2657939,MR1816158,MR1770033, MR3012911} for some of the very few contributions. Finite dimensional cases have been studied in \cite{MR2338444} and the references therein. On the other hand, the study of optimal control problems for variational inequalities (VIs) has been the subject of a number of recent studies; see, e.g., \cite{MR3033110,MR3232626,MR3484394} and the references therein. We note here that---to the best of our knowledge---the study of the stability of minimal and maximal solutions of QVIs and the optimal control thereof, with both being focus topics of this work, have not yet been treated in the literature. We further note that the stability of the solution set is also of relevance in identification problems involving QVIs; see \cite{MR3783774}.

The paper is organized as follows. In section \ref{section:ProblemFormulation} we introduce the optimal control problem associated to the QVI of interest, and we provide the mathematical foundation of the structure of spaces under consideration and their associated ordering. Additionally, in section \ref{examples} we study two classes of applications associated to impulse control problems and to QVIs arising as the coupling of VIs and nonlinear partial differential equations (PDEs), respectively. In section \ref{monotonicity}, we discuss the fundamental results due to Tartar that determine the existence of minimal and maximal solutions of the QVIs of interest. Abstract stability results from the operator theoretic point of view are the subject of section \ref{OperatorResults}, along with an example exploring limitations. In section \ref{monotonic}, we study minimal and maximal solutions under perturbations of the forcing term from below and from above. The paper ends in section \ref{control} which studies the well-posedness of the control problem for the QVI.

\subsection*{Notation}\label{section:Notation}

Throughout the paper we assume that $\Omega$ is an open subset of $\mathbb{R}^N$, and $L^p(\Omega)$ for $1\leq p\leq \infty$ denotes the usual Lebesgue space. For $\nu>0$, we define
\begin{equation*}
L_\nu^\infty(\Omega):=\{z\in L^\infty(\Omega): z(x)\geq \nu \:\:\text{ for almost all  (f.a.a.) }\:\: x\in \Omega\}.
\end{equation*}
Additionally, $H_0^1(\Omega)$ and $H^1(\Omega)$ denote  the usual Sobolev spaces; see \cite{MR2424078}.

For a Banach space $X$ we write $\|\cdot\|_X$ for a norm on $X$ and $X'$ for the topological dual of $X$ with $\langle\cdot,\cdot\rangle_{X',X}$ the associated duality pairing, respectively.
For a sequence $\{z_n\}_{n\in\mathbb{N}}$ in $X$ we denote its strong convergence to $z\in X$ by ``$z_n\to z$'' and weak convergence by ``$z_n\rightharpoonup z$''. Further, for two Banach spaces $X_1$ and $X_2$, we write  $\mathscr{L}(X_1,X_2)$ for the space of bounded linear operators from $X_1$ to $X_2$.

\section{A class of optimization problems with QVI constraints}\label{section:ProblemFormulation}

\subsection{Preliminaries}\label{sec:prelim}
Let $(V,H,V')$ be a Gelfand triple of Hilbert spaces, i.e., $V\hookrightarrow H \hookrightarrow V'$, where the embedding  $V\hookrightarrow H$ is dense and continuous, $H$ is identified with $H'$, and the embedding $H \hookrightarrow V'$ is dense and continuous as well (see \cite{MR0173945} and also, e.g., \cite{Denkowski2003-2}). Also, from now on we use $\langle\cdot,\cdot\rangle:=\langle\cdot,\cdot\rangle_{V',V}$ and $(\cdot,\cdot)$ for the inner product in $H$. 

Let further $H^+\subset H$ be a closed convex cone satisfying
$H^+=\{v\in H: (v,y)\geq 0 \text{ for all } y\in H^+\}$.
Note that $H^+$ defines the cone of non-negative elements inducing the vector ordering:
$x\leq y \text{ if and only if }y-x\in H^+$. 
Given $x\in H$, let $x^+$ denote the orthogonal projection of $x$ onto $H^+$, and {define $x^-:=x^+-x$}. Clearly, one has the decomposition $x=x^+-x^-\in H^+-H^+$ for every $x\in H$, and $(x^+,x^-)=0$. Further, the infimum and supremum of two elements $x,y\in H$ are defined as $\sup (x,y):=x+(y-x)^+$ and $\inf (x,y):=x-(x-y)^+$, respectively. The supremum of an arbitrary { completely ordered subset $R$} of $H$ that is bounded (in the order) above is also properly defined: { $R$ can be written as} $\{x_i\}_{i\in J}$, where $J$ is completely ordered, {  and it follows that} $\{x_i\}_{i\in J}$ is a generalized Cauchy sequence in $H$ (see \cite[Chapter 15, \S15.2, Proposition 1]{Aubin1979}); { its limit is the upper bound of the original set}. 
{Additionally, we have that} that norm convergence preserves order, i.e., if {$z_n\to z$ and $y_n\to y$ in $H$, then $z_n\leq y_n$ ($y_n-z_n\in H^+$) implies $z\leq y$, since $H^+$ is closed}. 
Also, we write $z_n\downarrow z$ in $H$ if $z_n\geq z_{n+1}$ for all $n\in \mathbb{N}$ and $z_n\to z$ in $H$, and analogously we consider $z_n\uparrow z$. Further, we have that if the sequence $\{z_n\}$ is non-increasing (non-decreasing) and bounded from below (above) in the sense of the order, then there exists $z\in H$ for which $z_n\downarrow z$ ($z_n\uparrow z$) in $H$. Now, concerning $V$ we assume that $y\in V$ implies $y^+\in V$, and that $(\cdot)^+:V\to V$ is a bounded operator, i.e., we have $M>0$ with $\|y^+\|_V\leq M\|y\|_V$ for all $y\in V$. {Note that this allows for $V=H^1(\Omega)$ but not $V=H^2(\Omega)$; see \cite[Theorem A.1, Chapter II, Appendix A]{Kinderlehrer}}

Given $x,y\in H$ such that $x\leq y$, we define the closed ``interval'' with $x$ and $y$ as its respective endpoints by $[x,y]:=\{z\in H: x\leq z \text{ and } z\leq y\}$. Furthermore, we write $[y,+\infty)$ and $(-\infty,y]$ instead of $\{z\in H: z\geq y\}$ and $\{z\in H:  z\leq y\}$, respectively.

Next we get more specific with respect to $V$ and $H$. In fact, both are assumed to be spaces of maps $h:\Omega\to \mathbb{R}$ over some open set $\Omega\subset\mathbb{R}^N$ with the following dense and continuous embedding: $L^\infty(\Omega)\hookrightarrow H$. Note that this implies that $L^\infty(\Omega)\hookrightarrow V'$, as well, { and  $L^\infty(\Omega)$ inherits the induced order on $H$}. Our prototypical example for this setting is $V:=H_0^1(\Omega)$ and $H:=L^2(\Omega)$ with $H^+:=L^2_+(\Omega)$, the closed convex cone of non-negative maps in $L^2(\Omega)$ with ``$v\leq w$'' for $v,w\in H$ iff $v(x)\leq w(x)$ almost everywhere (a.e.) on $\Omega$. Here, we have $v^+(x):=\max\{v(x),0\}$ 
for $x\in \Omega$.

Let $A:V\to V'$ be an operator (possibly nonlinear) that is 
\begin{itemize}
\item[-] homogenous of order one, i.e., $A(tu)=tA(u)$ for all $u\in V$, $t>0$;
\item[-] Lipschitz continuous, i.e., there exists $C>0$ such that
\begin{equation*}
\|A(u)-A(v)\|_{V'}\leq C\|u-v\|_V, \qquad \text{for all} \qquad u,v \in V;
\end{equation*}
\item[-] strongly monotone, i.e., there exists $c>0$ such that
\begin{equation*}
\langle A(u)-A(v),u-v\rangle \geq c\|u-v\|^2_V,  \qquad \text{for all} \qquad u,v \in V;
\end{equation*}
\item[-] T-monotone, i.e., 
\begin{equation*}
\langle A(u)-A(v),(u-v)^+\rangle \geq 0, \qquad \text{for all} \qquad u,v \in V,
\end{equation*}
where equality holds if and only if $(u-v)^+=0$. 
\end{itemize}
A well-known example for $A$ in the case $V=H_0^1(\Omega)$ (or $V=H^1(\Omega)$) and $H=L^2(\Omega)$ is given by the elliptic linear partial differential operator
\begin{equation}\label{eq:EllipticOperator}
    \langle Av,w\rangle=\sum_{i,j}\int_{\Omega}a_{ij}(x) \frac{\partial v}{\partial x_j}\frac{\partial w}{\partial x_i}\dif x+\sum_{i}\int_{\Omega}a_{i}(x) \frac{\partial v}{\partial x_i} w+ \int_{\Omega}a_0(x) vw \dif x,
\end{equation}
under suitable assumptions on $a_{ij},a_{i}$ and $a_0$ such as, e.g., $a_{ij}, a_0\in L^\infty(\Omega)$, $a_i\equiv 0$,  $\sum a_{ij}(x)\xi_j\xi_i\geq c |\xi|^2$ for all $\xi=\{\xi_i\}\in \mathbb{R}^N$, and $a_0(x)\geq \epsilon>0$ f.a.a. $x\in\Omega$. { Further, note that $u\mapsto Au+\max(u,0)$ is an example of a \emph{nonlinear} operator satisfying the above assumptions in this setting.}

For the definition of the constraint set of the QVI we need a map $\Phi$ with the following properties: {$\Phi:H^+\rightarrow H^+$ and $\Phi$ is increasing in $[0,+\infty)$, that is if $v,w\in [0,+\infty)$ and $v\leq w$ then $\Phi(v)\leq \Phi(w)$. The domain of $\Phi$ can then be extended to all $H$ by $x\mapsto\Phi(x^+)$ when needed.} Further properties of $\Phi$ will be specified below.


Next, we define the set-valued map $\mathbf{K}:H^+\to 2^V$ as
\begin{equation}\label{eq:Kobstacle}
    \mathbf{K}(\psi):=\{v\in V: v \leq \psi \}.
\end{equation}
Note that,  { for $\psi\in H^+$}, $\mathbf{K}(\psi)\subset V$ is non-empty ({ $0\in \mathbf{K}(\psi)$}), closed and convex. We also set $ \mathbf{K}(+\infty):=V$.

\subsection{Problem formulations}
The QVI problem of interest is the following one.\\[1ex]
\vspace{1mm}\noindent{\bf Problem $(\mathrm{P_{QVI}})$ }: Let $f\in V'$ be given.
\begin{equation}\label{eq:QVI}\tag{$\mathrm{P_{QVI}}$}
\text{Find } y\in \mathbf{K}(\Phi(y)): \langle A(y)-f,v-y\rangle \geq 0, \quad \forall v\in \mathbf{K}(\Phi(y)).
\end{equation}
This problem admits (in general) multiple solutions due to the non-convexity resulting from $y\mapsto \mathbf{K}(\Phi(y))$. Let $\mathbf{Q}(f)$ denote the associated solution set.

In applications, one is typically interested in confining the solution set $\mathbf{Q}(f)$ to a certain interval $[\underline{y},\overline{y}]$ for some given $\underline{y},\overline{y}\in H$. By considering $f$ a control force, the following optimal control problem arises in this setting:

\vspace{3mm}\noindent{\bf Problem $(\mathbb{P})$ }:
\begin{equation}\label{eq:CQVI}\tag{$\mathbb{P}$}
\begin{split}
&\text{minimize }  J(\mathbf{O},f):=J_1(T_{\sup}(\mathbf{O}),T_{\inf}(\mathbf{O}))+J_2(f)
\text{ over }(\mathbf{O},f)\in 2^H\times U,\\
&\text{subject to }  f\in U_{\text{ad}},\\
&\phantom{\text{subject to }} \mathbf{O}=\{z\in V: z \text{ solves \ref{eq:QVI}}\}.
\end{split}
\end{equation}
Here $U_{\mathrm{ad}}\subset U\subset V'$ is the set of admissible controls.
Moreover, $J_1:H\times H\to \mathbb{R}$ and  $J_2:U_{\mathrm{ad}}\to \mathbb{R}$, and for $\underline{y},\overline{y}\in H$ we define the map 
\begin{equation*}
T_{\sup}(\mathbf{O}):=
\left\{
  \begin{array}{ll}
    \sup_{z\in \mathbf{O}\cap [\underline{y},\overline{y}]} z, & \hbox{$\mathbf{O}\cap [\underline{y},\overline{y}]\neq \emptyset$ ,} \\
   \underline{y}, & \hbox{ otherwise,}
  \end{array}
\right.
\end{equation*}
and analogously 
\begin{equation*}
T_{\inf}(\mathbf{O}):=
\left\{
  \begin{array}{ll}
    \inf_{z\in \mathbf{O}\cap [\underline{y},\overline{y}]} z, & \hbox{$\mathbf{O}\cap [\underline{y},\overline{y}]\neq \emptyset$ ,} \\
   \overline{y}, & \hbox{ otherwise.}
  \end{array}
\right.
\end{equation*}
Problems of type \eqref{eq:CQVI} have not yet been considered in the literature and pose several formidable challenges. For instance, the proof of existence of a solution is highly delicate due to the dependence $y\mapsto \mathbf{K}(\Phi(y))$ and the fact that $y=y(f)$. As a consequence, the direct method of the calculus of variations is only applicable if certain convergence properties of that constraint set can be guaranteed. Another delicacy is related to the potential set-valuedness of the solution of the QVI in the constraint system of \eqref{eq:CQVI}. This fact requires to identify a suitable selection mechanism such as the one identifying the maximal or minimal solution, respectively, if available at all. We note, however, that in the special case where $T_{\inf}(\mathbf{Q}(f))$ and $T_{\sup}(\mathbf{Q}(f))$ also belong to $\mathbf{Q}(f)$, they are the minimal and maximal solution, respectively, to \eqref{eq:QVI} in $V\cap [\underline{y},\overline{y}]$ . Then the proof of existence of solutions to \eqref{eq:CQVI} reduces to a stability result for this minimal and maximal solution to the QVI of interest.

\section{Examples of applications}\label{examples}

Our work here is motivated by the following two application classes. The first one is associated to QVIs that result from coupling a variational inequality (VI) to a nonlinear partial differential equation (PDE). Such models have recently been studied in connection with thermoforming; see \cite{MR3903798}. The other problem class is given by QVIs arising in impulse control as pioneered by Bensoussan and Lions. We briefly describe both problem types in the sequel.

\subsection{QVIs arising from coupling VIs and nonlinear PDEs}\label{sec:VIsPDEs}

%


Consider the following class of compliant obstacle problems where the obstacle is given implicitly by solving a PDE, thus coupling a VI and a PDE. It consists in finding $(y,\Phi,z)\in V\times H\times W$ such that 
\begin{align}
 y \leq \Phi, \quad \langle A(y)-f, y-v \rangle &\leq 0, &\forall v \in V: v \leq \Phi,\label{eq:QVIforu}\\
\langle Bz+G(\Phi,y)-g, w\rangle&=0 &\forall w\in W,\label{eq:pdeForT}\\
\Phi &=Lz, &\text{in } H.\label{eq:mould}
\end{align}
Here, $V\hookrightarrow W\hookrightarrow H \hookrightarrow W'\hookrightarrow V'$, $f,g \in H^+$, $G:H\times H\to H$ is continuous and bounded, i.e., for some $M_G>0$, $\|G(\Phi,y)\|_H\leq M_G (\|\Phi\|_H+\|y\|_H)$, for all $(\Phi,y)\in H\times H$. Further, $L\colon W \to H$ is an increasing {affine} linear continuous map { with $L(0)\geq\nu>0$}. Additionally,  $B\in \L (W,W')$ is strongly monotone and satisfies $\langle Bz^-,z^+\rangle= 0$ for all $z\in W$ { (i.e., $B$ is T-monotone), and $A$ satisfies the assumptions of section \ref{sec:prelim}.}

Under mild conditions, the above problem can be cast into the form of \eqref{eq:QVI} as follows. Let $v\in H$, and consider the problem of finding $z\in W$ such that
\begin{align}
\langle Bz+G(\phi,v)-g, w\rangle&= 0 &\forall w\in W,\label{eq:mh1}\\
\phi &= Lz, &\text{in } H.\label{eq:mh2}
\end{align}
Assuming that for each $v\in H$, $z\mapsto G(Lz,v)$ is monotone, one can show the existence of a unique solution $z(v)\in W$ of \eqref{eq:mh1}--\eqref{eq:mh2}. Now set $\Phi(v):=\phi$.
Suppose additionally that $(G(Lz,y),z^-)\leq 0$ for all $z\in W$ and $y\in  H^+$. {Thus,} $z(v)\geq 0$ and $\Phi(v)=Lz(v)\geq {\nu}$ for all  $v\in H$. {In addition, if $v_1\leq v_2$ implies} 
\begin{equation*}
(G(Lv,v_1)-G(Lw,v_2),(v-w)^+){\geq} 0,
\end{equation*}
for all $w,v$, then $z(v_1)\leq z(v_2)$ and $\Phi(v_1)\leq \Phi(v_2)$,
as $L$ is increasing. This finally shows that \eqref{eq:QVIforu}--\eqref{eq:mould} has the form \eqref{eq:QVI} {with $\Phi$ as an increasing operator}.

In view of controlling the outcome of a stationary  industrial process one is clearly interested in forcing the solution set $\mathbf{Q}(f)$ to be a singleton which is close to a pre-specified desired state $y_d$. This can be modelled as follows. 

\begin{equation}\label{eq:ControlVIPDE}
\begin{split}
&\text{minimize }\frac{1}{2}\int_{\Omega}|T_{\sup}(\mathbf{Q}(f))-T_{\inf}(\mathbf{Q}(f))|^2+ \frac{1}{2}\int_{\Omega}|y_d-T_{\inf}(\mathbf{Q}(f))|^2+\frac{\lambda}{2}|f|_{U}^2\\
&\text{subject to } 0<\nu\leq f \leq F,\quad f\in U,\\
\end{split}
\end{equation}
for given $\lambda, \nu,F >0$. Here, $U$ denotes the underlying control space. Note that the first term in the objective aims at minimizing the distance between the maximal and minimal solution targeting single-valued $\mathbf{Q}(f)$, the second term aims at tracking $y_d$, whereas the last term 
associates an ($U$-) average cost of $\lambda$ to the control action. Notice that the smaller the $\lambda$, the cheaper the cost of the control gets and the smaller one expects the first two terms in the objective. Clearly, \eqref{eq:ControlVIPDE} fits the form of \eqref{eq:CQVI}.

\begin{example}\label{ex1}
A possible setting for this problem class satisfying all assumptions invoked so far is given by $V=H_0^1(\Omega)$, $H=L^2(\Omega)$,  $W=H^1(\Omega)$, with 
\begin{align*}
    \langle Ay,z\rangle&=\sum_{i,j}\int_{\Omega}a_{ij}(x) \frac{\partial y}{\partial x_j}\frac{\partial z}{\partial x_i}\dif x+ \int_{\Omega}a_0(x) yz \dif x, \qquad \forall y,z \in V, \\     \langle Bv,w\rangle&=\sum_{i,j}\int_{\Omega}b_{ij}(x) \frac{\partial v}{\partial x_j}\frac{\partial w}{\partial x_i}\dif x+ \int_{\Omega}b_0(x) vw \dif x, \qquad \forall v,w\in W,
\end{align*}
$a_{ij},b_{ij}, a_0, b_0\in L^\infty(\Omega)$, $\sum a_{ij}(x)\xi_j\xi_i\geq c |\xi|^2$ and $\sum b_{ij}(x)\xi_j\xi_i\geq c |\xi|^2$ for all $\xi=\{\xi_i\}\in \mathbb{R}^N$ and some $c>0$ and $a_0(x)\geq 0$ and $b_0(x)\geq \epsilon>0$ f.a.a. $x\in\Omega$. Additionally, for $y\geq 0$
\begin{equation*}
G(\Phi,y)=(\Phi-y)^+, \quad\text{and} \quad (Lz)(x)=k(x)z(x){+\nu}
\end{equation*}
with $k\in L^\infty(\Omega)^+$ and {$\nu>0$}. Further, $U=\mathbb{R}^M$ for some $M\in \mathbb{N}$, where $f=\sum_{m=1}^M f_m \chi_{\Omega_m}$, $f_m\in \mathbb{R}$ and $\Omega_m\subset \Omega$ for each $m$, and $\|f\|_U:=\|\{f_m\}\|_{\mathbb{R}^M}$. In this setting,  $\Phi:\{y\in H: 0\leq y\}\to H^+$ is non-decreasing, { and defined as $\Phi(z)=Lz$}.
\end{example}

\subsubsection{Thermoforming}
{
Thermoforming is a manufacturing process involving the heating of a membrane or plastic sheet to its pliable temperature and then forcing it (by means of positive or negative gas pressure) onto a mold of some metallic alloy. Subsequently, the membrane deforms and takes on the shape of the mold.  The process is used for large structures in the car industry but also to create microfluidic structures (e.g. channels on the range of micrometers). The amount of applications and the necessity of precision in a variety of settings has generated the need of  research into its modelling and accurate numerical simulation; see \cite{Munro2001}, \cite{Warby2003}, \cite{karamanou2006computational}.

We follow closely \cite{MR3903798} and consider the following modeling assumptions: i) The temperature for the membrane is always a constant value ii) The mold grows in an affine way with respect to changes in its temperature iii) The temperature of the mold is subject to diffusion, convection and boundary conditions arising from the insulated boundary and it depends on the vertical distance between the mold and the membrane. 

The above means that the problem can be formulated as \eqref{eq:QVIforu}-\eqref{eq:mould} where $y$ denotes the position of the membrane, $z$ the temperature, and $\Phi$ the mold. Specifically, $V=H_0^1(\Omega),W=H^1(\Omega),$ and $H=L^2(\Omega)$, meaning that we assume that the membrane is clamped on the boundary, and we assume thermal insulation there as well. The operators $A$ and $B$ are second order elliptic related to the elastic deformation of the membrane, and the diffusion of the temperature, respectively.  Further, $G(\Phi,y)=g(\Phi-y)$ where $g(x)=-1$ if $x\leq 0$, $g(x)=x-1$ if $0<x\leq 1$ and $g(x)=0$ if $1<x$, $(Lz)(x)=l(x)z(x)+\Phi_0(x)$ for some non-negative $l$, and where $\Phi_0$, the shape of the mold at the baseline temperature.

}

\subsection{Impulse control}
We consider impulse control problems (see \cite{MR673169}) for the following stochastic differential equation
\begin{equation*}
du=b(u)dt+\sigma(u)dw(t),\qquad u(0)=x\in \mathbb{R}^N,
\end{equation*}
where $b,\sigma:\mathbb{R}^N\to \mathbb{R}^N$ are Lipschitz functions whose regularity will be specified later. Let $a_{ij}:=\sigma_i\sigma_j^\top/2$.

The control is carried on instances $0\leq\theta_1\leq \theta_2\leq \cdots$, and the system is forced from $y(\theta_n^-)$ to $y(\theta_n^-)+\xi_n$ on the instance $\theta_n$. The energy associated to the process is given by the expected value
\begin{equation*}
\mathbb{E} \left\{\int_0^{\tau_x} f(u(t))e^{-\int_0^t a_0(u(s))\mathrm{d}s}\mathrm{d}t+\sum_{n}(k+c_0(\xi_n)) e^{-\int_0^{\theta_n} a_0(u(s))\mathrm{d}s}\chi_{\theta_n<\infty}\right\}=:\mathcal{H}(f,w,x),
\end{equation*}
where $w=\{(\theta_n,\xi_n)\}_{n=1}^\infty$, and with $\tau_x:=\inf\{t: u(t^-)\notin \Omega \:\text{ or }\: u(t)\notin \Omega \}$, for some open $\Omega\subset \mathbb{R}^N$. In this setting, $f$ uniquely determines the value function
\begin{equation*}
\min_{w\in \mathcal{W}} \mathcal{H}(f,w,x),
\end{equation*}
which represents the cost of the optimal control associated to the initial condition $x$ and the cost function $f$. Here, $\mathcal{W}$ is the set of all possible instances and jumps $\{(\theta_n,\xi_n)\}_{n=1}^\infty$. The optimization of the above quantity via choosing $f$ turns out to be of interest. Indeed, in specific applications $f$ determines the value of a certain stock or energy type per unit of time. The goal is then to solve
\begin{equation}\label{eq:Opt}
\text{minimize } \int_{\Omega}\left(s-\min_{w\in \mathcal{W}} \mathcal{H}(f,w,x)\right)^2\mathrm{d}x+\frac{\lambda}{2}|f|_U^2\quad\text{subject to }f\in U_{\mathrm{ad}},
\end{equation}
where $U_{\mathrm{ad}}\subset U$ is the set of admissible functions $f$,  $|f|_U^2$ represents cost of the choice of $f$, $\lambda>0$ is a weight, and $s\geq 0$ is a desired average cost that could be zero.

\subsubsection{Bounded case} \label{subsub1} We consider $\Omega$ bounded with a sufficiently smooth boundary, with $V=H^1(\Omega)$, $H=L^2(\Omega)$, and $H^+=L_+^2(\Omega)$, where $A$ is of the type \eqref{eq:EllipticOperator} with 
\begin{align*}
a_{ij}&=a_{ji}\in W^{1,\infty}(\mathbb{R}^N),\:\:\:\: 1\leq i,j\leq N, \:\:\: \sum a_{ij}\xi_i\xi_j\geq \alpha |\xi|^2, \:\:\:\: \alpha>0, \forall \xi\in \mathbb{R}^N, \\
a_i,a_0&\in L^\infty(\mathbb{R}^N),\:\:\:\:\:\:      b_i=-a_i+\sum_{j}\frac{\partial a_{ij}}{\partial x_j}\in W^{1,\infty}(\mathbb{R}^N),\\
a_0(x)&\geq r>0, \:\:\:\:\:\: \text{f.a.a.}\:\:\:\: x\in \Omega.
\end{align*}
{Further suppose that the coefficients are such that the induced $A$ is strongly monotone. Note that order one homogeneity, Lipschitz continuity, and T-monotonicity already hold; the latter follows immediately as $\langle A z^-,z^+\rangle =0$ for all $z$}. Consider 
\begin{equation*}
(\Phi y)(x):=k+\mathrm{essinf}_{x+\xi \in \overline{\Omega}} (c_0(\xi)+y(x+\xi)),
\end{equation*}
where $c_0\in C(\mathbb{R}^N_+, \mathbb{R})$ is such that $c_0(0)=0$ is sub-linear and non-decreasing, with $f\in L^p(\Omega)$ with $p>N$ and $f\geq 0$. 

In this setting one can show that the solution set  $\mathbf{Q}(f)$ of \eqref{eq:QVI} is a singleton, $\mathbf{Q}(f)=\{y^*\}$, and $y^*$ determines the value function of the impulse control problem of interest (see \cite{MR673169}), i.e.,
\begin{equation*}
y^*(x)=\min_{w\in \mathcal{W}} \mathcal{H}(f,w,x),
\end{equation*}
a.e. for $x\in \Omega$. Hence, problem \eqref{eq:Opt} amounts to controlling the solution to the quasi-variational inequality and is, thus, of the form \eqref{eq:CQVI}.

\subsubsection{Unbounded case}\label{sec:unboundedCase}
Let $\omega(x):=\exp(-\mu\sqrt{1+|x|^2})$ for $x\in\mathbb{R}^N$, and consider the weighted spaces $V=H^1(\mathbb{R}^N,\omega)$, and $H=L^2(\mathbb{R}^N,\omega)$ with  $H^+=L_+^2(\mathbb{R}^N,\omega)$ the usual cone of non-negative maps.  In particular, $L^2(\mathbb{R}^N,\omega)$ is the space of (equivalence classes of) measurable functions $h:\mathbb{R}^N\to \mathbb{R}$ for which $|h|^2_{L^2(\mathbb{R}^N,\omega)}:=\int_{\mathbb{R}^N} |h(x)|^2 \omega(x)^2\mathrm{d}x<+\infty$, and $H^1(\mathbb{R}^N,\omega)$ is the space of (equivalence classes of) of functions $g:\mathbb{R}^N\to \mathbb{R}$ for which $g$ and its weak gradient $\nabla g$ belong to $L^2(\mathbb{R}^N,\omega)$ and $L^2(\mathbb{R}^N,\omega)^N $, respectively. 

The operator $A:V\to V'$ is given by
\begin{equation*}
\langle Av,y\rangle=\sum_{i,j}\int_{\mathbb{R}^N}a_{ij} \frac{\partial v}{\partial x_j}\frac{\partial y}{\partial x_i}\omega^2\dif x+\sum_{i}\int_{\mathbb{R}^N}a_i\frac{\partial v}{\partial x_j} y \omega^2\dif x + a_0\int_{\mathbb{R}^N} vy \omega^2\dif x,
\end{equation*}
with $a_{ij}, a_i, b_i$ as in Section \ref{subsub1}, but with $a_0(x) = r$ for all $x$ and a real $r>0$.

Define
\begin{equation*}
U:=\{f:\mathbb{R}^N\to \mathbb{R}: f \text{ is measurable and } 0\leq f(x)\leq C(1+|x|^s), \:\:\:\: {\forall} x\in\mathbb{R}^N\},
\end{equation*}
where the bound above holds for some $C>0$ {and some $s\geq 0$,}
and the map $\Phi$ by
\begin{equation*}
(\Phi y)(z):=k+\mathrm{essinf}_{\xi \geq 0} (c_0(\xi)+y(z+\xi)),
\end{equation*}
where $k>0$, and $c_0\in C(\mathbb{R}^N_+, \mathbb{R})$, with $c_0(0)=0$, is sub-linear, non-decreasing with $\lim_{|\xi|\to+\infty}c_0(\xi)=+\infty$ and for which $c_0(\xi)\leq a |\xi|^\gamma$ for some $a,\gamma>0$.

In this scenario, the set of solutions $\mathbf{Q}(f)$ of \eqref{eq:QVI} is  not necessarily a singleton, and both $T_{\inf}(\mathbf{Q}(f))$ and $T_{\sup}(\mathbf{Q}(f))$ have probabilistic interpretations associated to the value function in impulse control. In particular,
\begin{equation*}
T_{\inf}(\mathbf{Q}(f))(x)=\min_{w\in \mathcal{W}} \mathcal{H}(f,w,x),
\end{equation*}
i.e., $T_{\inf}(\mathbf{Q}(f))$ is the value function associated with the initial impulse control problem. Then \eqref{eq:Opt} has the form of \eqref{eq:CQVI} for appropriate choices of $J_1$ and $J_2$.

\section{Increasing maps and QVI solutions}\label{monotonicity}




This section is strongly related to a result due to Tartar \cite{Tartar74}; see also {\cite[Chapter 15]{Aubin1979}}. Upon converting $(\mathrm{P_{QVI}})$ into a fixed-point equation, the corresponding approach yields the existence of a solution for an increasing fixed-point map under very mild assumptions. We note that the technique is analogous to the one by Kolodner and Birkhoff; see \cite{Kolodner68,Baiocchi,Birkhoff61}. 

We start by recalling Tartar's result (compare \cite[Chapter 15, \S15.2]{Aubin1979}) which rests on increasing maps. In this vein, we call $T:H\rightarrow H$ increasing iff $v\leq w$ implies $T(v)\leq T(w)$.

\begin{thm}[\textsc{Birkhoff-Tartar}]\label{thm:BirkhoffTartar} 
Suppose $T:H\rightarrow H$ is an increasing map and let $\underline{y}$ be a sub-solution and $\overline{y}$ be a  super-solution of the map $T$, that is:
\begin{equation*}
\underline{y}\leq T(\underline{y}) \quad \text{ and } \quad T(\overline{y})\leq \overline{y}.
\end{equation*}
If  $\underline{y}\leq  \overline{y}$, then the set of fixed points of the map $T$ in the interval $[\underline{y}, \overline{y}]$ is non-empty and has a smallest and a largest element.
\end{thm}


%

We apply the above result to $(\mathrm{P_{QVI}})$ and first need to introduce the following VI.

\vspace{1mm}\noindent{\bf{Problem $(\mathrm{P_{VI}})$}}: Let $\psi\in H^+$, $f\in V'$ be given.
\begin{equation}\label{eq:VI}\tag{$\mathrm{P_{VI}}$}
\text{Find } y\in \mathbf{K}(\psi): \langle A{(y)}-f,v-y\rangle \geq 0, \quad \forall v\in \mathbf{K}(\psi).
\end{equation}

The solution to \eqref{eq:VI} can be proven to be unique by standard methods. For $(f,\psi)\in (V',H^+)$, we denote the unique solution to  \eqref{eq:VI} as $S(f,\psi)$. Before we can make use of Theorem \ref{thm:BirkhoffTartar}, we state the following property of the map $(f,\psi)\mapsto S(f,\psi)$ {that relies on the fact that $A$ is T-monotone}. Its proof can be found in  {\cite[Section 4:5, Theorem 5.1]{Rodrigues1987}}.

\begin{prop}\label{Prop:IncreasingSolMap}
Let $f_1,f_2\in V'$ and $\psi_1,\psi_2\in H^+$ be such that $f_1\leq f_2$ in $V'$ and $\psi_1\leq\psi_2$. Then it holds that
$S(f_1,\psi_1)\leq S(f_2,\psi_2)$.
\end{prop}


We note that in the above result $f_1\leq f_2$ in $V'$ is well-defined, since {$V'$} inherits the order in $H$, so that $f_1\leq f_2$ iff $\langle f_2-f_1,v\rangle\geq 0$ for all $v\in V$ such that $v\geq 0$. Further observe that the case $\psi=+\infty$ is also allowed, where $S(f,+\infty)$ denotes the solution of the unconstrained problem 
\begin{equation*}
\text{ Find } y\in V \text{such that } \langle A(y),v\rangle=\langle f,v\rangle, \quad \text{for all }v\in V.
\end{equation*}
This implies that
\begin{equation*}
    S(f, \psi)\leq S(f, +\infty), \quad \forall f\in V', \psi \in H^+.
\end{equation*}

In order to apply the Birkhoff-Tartar Theorem to the QVI problem of interest, we need to identify a proper interval $[\underline{y},\overline{y}]$, with $\underline{y}$ a sub-solution and $\overline{y}$ a super-solution of the map $y\mapsto S(f,\Phi(y))$. In our case, we choose $\underline{y}=0$, since we infer from Proposition \ref{Prop:IncreasingSolMap} that \begin{equation}\label{eq:subsol}
0=S(0,\Phi(0))\leq S(f, \Phi(0)),
\end{equation}
for any $f\geq 0$ in $V'$. On the other hand, we assume that $f\in U_{\mathrm{ad}}\subset V'$ is bounded from above (in the $V'$-order) by some $F$. Then let $\overline{y}=S(F,+\infty)$, for which
\begin{equation}\label{eq:supsol}
S(f,\Phi(\overline{y})) \leq S(F,+\infty)= \overline{y}.
\end{equation}
This leads to the following result.


\begin{thm}[\textsc{Tartar}]\label{thm:Tartar} Let $U_{\mathrm{ad}}\subset \{f\in V': 0\leq f\leq F\}$ for some  $0\leq F\in V'$. Then, there are $\underline{y},\overline{y}$ such that  for each $f\in U_{\mathrm{ad}}$, the set of fixed points of the map $y\mapsto S(f,\Phi(y))$ in the interval $[\underline{y},\overline{y}]$ is non-empty and contains a smallest and a largest element, i.e., there are fixed points $y^*_{\mathrm{min}}$ and $y^*_{\mathrm{max}}$ in $V$ such that
\begin{equation*}
    \mathbf{Q}(f)\cap[\underline{y}, \overline{y} ] =\mathbf{Q}(f)\cap \left[y^*_{\mathrm{min}}, y^*_{\mathrm{max}} \right]\neq \emptyset.
\end{equation*}
\end{thm}

In light of Theorems \ref{thm:BirkhoffTartar} and \ref{thm:Tartar}, there exist operators $\mathbf{m}$ and $\mathbf{M}$, which map an increasing map on the interval $[\underline{y}, \overline{y} ]$ to its minimal and maximal fixed points, respectively; insofar that sub- and super-solutions $\underline{y}$ and $\overline{y} $ exist.

We fix some notation now. In the case of a general increasing map $T$, with sub- and super-solutions $\underline{y}$ and $\overline{y}$, respectively, we denote by $\mathbf{m}(T)$ and $\mathbf{M}(T)$ its minimal and maximal fixed points in some interval $[\underline{y}, \overline{y} ]$. When the map $T$ is given by $y\mapsto S(f,\Phi(y))$ for some $f$, we write $\mathbf{m}(f)$ and $\mathbf{M}(f)$. In particular, it follows that if $\mathbf{Q}(f)$ is the set of solutions of \eqref {eq:QVI}, then 
\begin{equation*}
T_{\sup}(\mathbf{Q}(f))=\mathbf{M}(f), \qquad \text{and}\qquad T_{\inf}(\mathbf{Q}(f))=\mathbf{m}(f),
\end{equation*}
where $T_{\sup}, T_{\inf}$ are given in \eqref{eq:CQVI}.

For an operator $T$ as in Theorem \ref{thm:BirkhoffTartar}, the fixed points $\mathbf{m}(T)$ and $\mathbf{M}(T)$ are determined (see {\cite[Proposition 2, Ch. 15, \S 15.2]{Aubin1979}} for a proof) by the maximal and minimal elements of the sets $Z(T)$ and $\tilde{Z}(T)$, respectively, where
\begin{align*}
            &Z(T)=\{x\in X(T): x\leq y \text { for all } y\in Y(T)\},\\
             &\tilde{Z}(T)=\{y\in Y(T): x\leq y \text { for all } x\in X(T)\},
\end{align*}
and
\begin{align*}
    &X(T)=\{x\in H: x\in [\underline{y}, \overline{y}] \text{ and } x\leq T(x)\}, \\
        &Y(T)=\{x\in H: x\in [\underline{y}, \overline{y}] \text{ and } x\geq T(x)\}.
\end{align*}
In the following section, we use this setting for $\mathbf{m}(T)$ and $\mathbf{M}(T)$ to establish stability results. We also provide an equivalent definition that is exploited subsequently.


\section{Stability results}\label{OperatorResults}

For the existence of optimal controls for our problem of interest, we need to study the stability of the maps $f\mapsto \mathbf{m}(f)$ and $f\mapsto \mathbf{M}(f)$. In the general case of an increasing map $T$, we now prove that $\mathbf{m}(T)$ and $\mathbf{M}(T)$ are stable from below and above, respectively, provided $T$ has certain complete continuity properties. 



\begin{thm}\label{thm:approximationMinMax}
Let $T, R_n, U_n :H\rightarrow V \subset H$ be increasing mappings with $n\in \mathbb{N}$. Assume further:
\begin{itemize}
\item[(i)] $T:V\rightarrow V$ is completely continuous with respect to monotone sequences, i.e., if $v_n\rightharpoonup v$ in $V$ and $v_n \leq v_{n+1}$ (or $v_n \geq v_{n+1}$) for all $n\in\mathbb{N}$, then $T(v_n)\rightarrow T(v)$ in $V$.
\item[(ii)] Sets of fixed points of $T, R_n, U_n$ (assuming they exist) are uniformly bounded in $V$ with respect to $n\in \mathbb{N}$, and  that 
\begin{equation*}
   \underline{y}\leq  R_n(v)\leq R_{n+1}(v)\leq T (v)\leq U_{n+1}(v)\leq U_n(v)\leq \overline{y}, \quad \forall v\in [\underline{y}, \overline{y}],  n\in \mathbb{N},
\end{equation*}
for some $\underline{y}$ and $\overline{y}$ in $V$ such that $\underline{y}\leq \overline{y}$ .
\item[(iii)] If $\{v_n\}$ and $\{w_n\}$ are bounded sequences in $V$ such that $ v_n\leq v_{n+1}\leq \overline{y}$ and $ w_n\geq w_{n+1}\geq \underline{y}$, then 
\begin{align*}
    &\lim_{n\rightarrow\infty}\|R_n(v_n)-T (v_n)\|_{V}=0 &\text{ and } & &\lim_{n\rightarrow\infty}\|U_n(w_n)-T (w_n)\|_V=0.
\end{align*}
\end{itemize}
Let $\mathbf{m}$ and  $\mathbf{M}$ be the operators that take an increasing map with sub- and supersolutions $[\underline{y}, \overline{y}]$ into the minimal and maximal solutions of Theorem \ref{thm:BirkhoffTartar}, respectively.
Then
\begin{align*}
    \mathbf{m}(R_n)\rightarrow \mathbf{m}(T) \text{ in } V, \qquad\text{ and }\qquad
        \mathbf{M}(U_n)\rightarrow \mathbf{M}(T) \text{ in } V,
\end{align*}
and
\begin{align*}
    \mathbf{m}(R_n)\uparrow \mathbf{m}(T) \text{ in } H, \qquad\text{ and }\qquad
        \mathbf{M}(U_n)\downarrow \mathbf{M}(T) \text{ in } H,
\end{align*}
as $n\rightarrow\infty$, respectively.
\end{thm}

\begin{proof}
First note that since $\underline{y}\leq R_n(v)\leq T (v)\leq U_n(v) \leq \overline{y}$, the operators $\mathbf{m}$ and $\mathbf{M}$ are well defined  on $T, R_n$ and $U_n$ for each $n\in \mathbb{N}$ since each of these maps is increasing with the same sub- and supersolutions. Consider the sets
\begin{align*}
    &X(T)=\{x\in H: x\in [\underline{y}, \overline{y}] \text{ and } x\leq T(x)\}, \\
        &Y(T)=\{x\in H: x\in [\underline{y}, \overline{y}] \text{ and } x\geq T(x)\},\\
            &Z(T)=\{x\in X(T): x\leq y \text { for all } y\in Y(T)\};
\end{align*}
and similarly for each $R_n$ and $U_n$, $n\in\mathbb{N}$.

Since $R_n(v)\leq R_{n+1}(v)\leq T(v)$ for all $v\in [\underline{y}, \overline{y}]$ it follows that
\begin{equation}\label{eq:IneqSets1}
    X(R_n)\subset X(T) \text{ and } Y(T)\subset Y(R_n), \text { and hence } Z(R_n)\subset Z(T),
\end{equation}
and also
\begin{equation}\label{eq:IneqSets2}
    X(R_n)\subset X(R_{n+1}) \text{ and } Y(R_{n+1})\subset Y(R_n), \text { and hence } Z(R_n)\subset Z(R_{n+1}).
\end{equation}
Clearly $Z(R_n)$ and $Z(T)$  are not empty, since $\underline{y}$ belongs to both of them. Following the proof of Tartar's Theorem (compare {\cite[Proposition 2, Ch. 15, \S 15.2]{Aubin1979}}) we observe that $\mathbf{m}(R_n)$ and $\mathbf{m}(T)$ correspond to the maximal elements of $Z(R_n)$ and $Z(T)$, respectively. Consequently, it follows from \eqref{eq:IneqSets1} and \eqref{eq:IneqSets2} that
\begin{equation}\label{RnIneq}
\mathbf{m}(R_n)\leq \mathbf{m}(R_{n+1})\leq\mathbf{m}(T), \quad \forall n\in \mathbb{N}.
\end{equation}
Hence, $\{\mathbf{m}(R_n)\}$ is a monotonically increasing sequence which is bounded from above (for the ordering ``$\leq$''), which implies that $\mathbf{m}(R_n)\to \hat{y}$ in $H$, for some $\hat{y}\in H$.  We also know that the sets of fixed points of the maps are uniformly bounded in $V$. Therefore, we infer $\mathbf{m}(R_{n})\rightharpoonup  \hat{y}$ in $V$, that the sequence is non-decreasing, and hence $T(\mathbf{m}(R_{n}))\rightarrow T(\hat{y})$ in $V$. Since $\mathbf{m}(R_{n})=R_{n}(\mathbf{m}(R_{n}))$ and $$\lim_{j\rightarrow\infty}\|R_{n}(\mathbf{m}(R_{n}))-T(\mathbf{m}(R_{n}))\|_V=0,$$ we have $R_{n}(\mathbf{m}(R_{n}))\rightarrow T(\hat{y})$. Therefore, $\mathbf{m}(R_{n})\rightarrow T(\hat{y})$, but since $\mathbf{m}(R_{n})\rightharpoonup  \hat{y}$ it follows that $\mathbf{m}(R_{n})\rightarrow \hat{y}$ in $V$, where $\hat{y}$ is a fixed point of $T$. 

Since $\mathbf{m}(R_{n})\leq \mathbf{m}(T)$ for all $n$, we have $\hat{y}\leq \mathbf{m}(T)$. However,  $\mathbf{m}(T)$ is the minimal fixed point of $T$, and therefore $\hat{y}= \mathbf{m}(T)$. Summarizing we have
\begin{equation*}
\mathbf{m}(R_n)\rightarrow \mathbf{m}(T) \text{ in } V \quad\text{ and }\quad \mathbf{m}(R_n)\uparrow \mathbf{m}(T) \:\text{ in  } H.  
\end{equation*}

Now we consider the upper bound. We define
\begin{equation*}
    \tilde{Z}(T):=\{y\in Y(T): x\leq y \text { for all } x\in X(T)\},
\end{equation*}
and analogously for $U_n$, $n\in\mathbb{N}$.
Since $T(v)\leq U_{n+1}(v)\leq U_n(v)$ for all $v\in [\underline{y}, \overline{y}]$ and $n\in\mathbb{N}$ it follows that
\begin{equation}\label{eq:IneqSets3}
    X(T)\subset X(U_n) \text{ and } Y(U_n)\subset Y(T), \text { and hence } \tilde{Z}(U_n)\subset \tilde{Z}(T),
\end{equation}
and also 
\begin{equation}\label{eq:IneqSets4}
    X(U_{n+1})\subset X(U_n) \text{ and } Y(U_n)\subset Y(U_{n+1}) \text { hence } \tilde{Z}(U_n)\subset \tilde{Z}(U_{n+1}).
\end{equation}

Clearly, $\overline{y}\in \tilde{Z}(T), \tilde{Z}(U_n)$ and then, as before, we apply Zorn's Lemma (with the reverse order) to find minimal elements $\mathbf{M}(T)$ and $\mathbf{M}(U_n)$, such that  $$\mathbf{M}(T)\leq \mathbf{M}(U_{n+1})\leq \mathbf{M}(U_n)\leq \overline{y}.$$ 
Then, $\{-\mathbf{M}(U_n)\}$ is a monotonically increasing sequence which is bounded above for the ordering ``$\leq$''. This implies that $\mathbf{M}(R_n)\to \check{y}$ in $H$ for some $\check{y}\in H$. Since $\{\mathbf{M}(U_n)\}$ is also uniformly bounded in $V$, we have $\mathbf{M}(U_{n})\rightharpoonup \check{y}$ and this latter sequence is also non-increasing. Therefore, we infer $T(\mathbf{M}(U_{n}))\rightarrow T(\check{y})$. Since $\mathbf{M}(U_{n})=U_{n}(\mathbf{M}(U_{n}))$  and $$\lim_{j\rightarrow\infty}\|U_{n}(\mathbf{M}(U_{n}))-T(\mathbf{M}(U_{n})\|_V=0,$$ we get $U_{n}(\mathbf{M}(U_{n}))\rightarrow T(\check{y})$ and $\mathbf{M}(U_{n})\rightarrow \check{y}$, both in $V$, where $\check{y}=T(\check{y})$.
 
As in the previous case, since  $\mathbf{M}(T)\leq \mathbf{M}(U_{n})$, we have that $\mathbf{M}(T)\leq \check{y}$. However, $\mathbf{M}(T)$ is the maximal fixed point to $T$, and therefore $\check{y}= \mathbf{M}(T)$. Hence, we have
\begin{equation*}
\mathbf{M}(U_n)\rightarrow \mathbf{M}(T) \text{ in } V \quad\text{ and }\quad \mathbf{M}(U_n)\downarrow \mathbf{M}(T) \text{ in } H, 
\end{equation*}
which ends the proof.
\end{proof}

This result is sharp regarding lower and upper approximations, as it is generally  not possible to obtain  $\mathbf{M}(R_n)\rightarrow \mathbf{M}(T)$ and $\mathbf{m}(U_n)\rightarrow \mathbf{m}(T)$. We illustrate this fact by means of the following one dimensional example.
\begin{ex}\label{ex:Counterexample}
 Let $T:[0,1]\rightarrow [0,1]$ be defined as
\begin{equation*}
    T(v)=\left\{
           \begin{array}{ll}
             a, & \hbox{$0\leq v <a$;} \\
             v, & \hbox{$a\leq v < b$;} \\
             b, & \hbox{$b\leq v \leq 1$ .}
           \end{array}
         \right.
\end{equation*}
with $0<a<b<1$ and where $\mathbf{m}(T)=a$ and $\mathbf{M}(T)=b$ and
\begin{align*}
    R_n(v)&=\left\{
           \begin{array}{ll}
             a, & \hbox{$0\leq v <\frac{1}{n}$,} \\
             T(v-\frac{1}{n}), & \hbox{$\frac{1}{n}\leq v \leq 1$,}
           \end{array}
         \right.
& U_n(v)=\left\{
           \begin{array}{ll}
             T(v+\frac{1}{n}), & \hbox{$0\leq v <1-\frac{1}{n}$,} \\
             b, & \hbox{$1-\frac{1}{n}\leq v \leq 1$ .}
           \end{array}
         \right.
\end{align*}
Suppose that $n> N$ such that $\frac{1}{N}\leq a$ and $b\leq1-\frac{1}{N}$. Then, all the assumptions of the previous theorem hold, but $\mathbf{m}(R_n)=\mathbf{M}(R_n)=a$ and $\mathbf{m}(U_n)=\mathbf{M}(U_n)=b$ and hence $a=\mathbf{M}(R_n)\rightarrow \mathbf{M}(T)=b$ and $b=\mathbf{m}(U_n)\rightarrow \mathbf{m}(T)=a$ only hold for $a=b$, a contradiction.
\end{ex}

Although, as observed in the previous example, a general approximation theorem (under the hypotheses of Theorem \ref{thm:approximationMinMax}) for minimal and maximal fixed points seems elusive, we establish such a result for the specific case of the QVIs of interest. In order to achieve this, we first determine an equivalent definition of $\mathbf{m}$ and $\mathbf{M}$ but from slightly different means as in the Birkhoff-Tartar Theorem (see {\cite[Proposition 2, Ch. 15, \S 15.2]{Aubin1979}}). 

\begin{lem}
Let $T:H\rightarrow H$ be an increasing map with sub-solution $\underline{y}$ and super-solution $\overline{y}$  such that $\underline{y}\leq \overline{y}$. Then $\mathbf{m}(T)$, the maximal element of $Z(T)$, can also be defined as the maximal element of the set $Z^\bullet(T)$, which is defined as follows
\begin{align*}
    &X(T)=\{x\in H: x\in [\underline{y}, \overline{y}] \text{ and } x\leq T(x) \}, \\
        &Y^\bullet (T)=\{x\in H: x\in [\underline{y}, +\infty) \text{ and } x\geq T(x)\},\\
            &Z^\bullet(T)=\{x\in X(T): x\leq y \text { for all } y\in Y^\bullet(T)\}.
\end{align*}
Similarly, $\mathbf{M}(T)$, the minimal element of $\tilde{Z}(T)$, can also be defined as the minimal element of the set $\tilde{Z}^\bullet(T)$, defined as
\begin{align*}
    &X^\bullet(T)=\{x\in H: x\in (-\infty, \overline{y}] \text{ and } x\leq T(x) \}, \\
        &Y (T)=\{x\in H: x\in [\underline{y}, \overline{y}] \text{ and } x\geq T(x)\},\\
            &\tilde{Z}^\bullet(T)=\{y\in Y(T): x\leq y \text { for all } x\in X^\bullet(T)\}.
\end{align*}
\end{lem}
\begin{proof}
We begin by noting that $\mathbf{m}(T)$ is the maximal element of $Z(T)$, and $\mathbf{M}(T)$ is the minimal element of $\tilde{Z}(T)$, as shown in the proof of the Birkhoff-Tartar Theorem (see {\cite[Proposition 2, Ch. 15, \S 15.2]{Aubin1979}}).

Since $\underline{y}\in Z^\bullet(T)$ and $Z^\bullet(T)\subset [\underline{y}, \overline{y}]$, $Z^\bullet(T)$ is nonempty and bounded in $H$, we may apply Zorn's Lemma (see {\cite[Proposition 1, Ch. 15, \S 15.2]{Aubin1979}}). Let $x^*\in Z^\bullet(T)$ be {a} maximal element of $Z^\bullet(T)$. It follows from $Y(T)\subset Y^\bullet (T)$ that $Z^\bullet(T)\subset Z(T)$. Therefore
\begin{equation*}
   x^*\leq \mathbf{m}(T),
\end{equation*}
where $\mathbf{m}(T)$ is the maximal element of $Z(T)$ and the minimum fixed point of $T$ in $[\underline{y}, \overline{y}]$.

Since $x^*\in Z^\bullet(T)$, it follows by definition that $x^*\in X(T)$. Hence, we have $\underline{y}\leq x^*\leq \overline{y}$ and $x^*\leq T(x^*)$. Also, since $T$ is an increasing map, it holds that $\underline{y}\leq T(\underline{y})\leq T(x^*)\leq T(\overline{y})\leq \overline{y}$ and $T(x^*)\leq T(T(x^*))$, i.e., $T(x^*)\in X(T)$. Furthermore, if $y\in Y^\bullet (T)$, then $x^*\leq y$. Hence, $T(x^*)\leq T(y)\leq y$, i.e., $T(x^*)\in  Z^\bullet(T)$. However, $x^*\in Z^\bullet(T)$ is maximal, i.e., $T(x^*)\leq x^*$. Consequently, $x^*=T(x^*)$ and $x^*\in [\underline{y}, \overline{y}] $. Finally, $\mathbf{m}(T)$ is the minimal fixed point of $T$ in $ [\underline{y}, \overline{y}]$ so that
\begin{equation*}
    \mathbf{m}(T)\leq x^*.
\end{equation*}

{Hence from the above $\mathbf{m}(T)= x^*$}. Noting that $\tilde{Z}^\bullet(T)\subset [\underline{y}, \overline{y}]$ and $ \overline{y}\in \tilde{Z}^\bullet(T)$, we can once again apply Zorn's Lemma (with the reversed order). Let $x^*$ be {a} minimal element of $ \tilde{Z}^\bullet(T)$.  We have that $X(T)\subset X^\bullet(T)$ which implies $\tilde{Z}^\bullet(T)\subset \tilde{Z}(T)$. Therefore, it holds that
\begin{equation*}
    \mathbf{M}(T)\leq x^*,
\end{equation*}
where $\mathbf{M}(T)$ is the minimum element of $\tilde{Z}(T)$ and the maximum fixed point of $T$ in $[\underline{y}, \overline{y}]$.

Since $x^*\in \tilde{Z}^\bullet(T)$, we have by definition that $x^*\in Y(T)$, i.e., $\underline{y}\leq x^*\leq \overline{y}$ and $T(x^*)\leq x^*$. Furthermore, the map $T$ is increasing and therefore $\underline{y}\leq T(\underline{y})\leq T(x^*)\leq T(\overline{y})\leq \overline{y}$ and $T(T(x^*))\leq T(x^*)$, i.e., $T(x^*)\in Y(T)$.

For an arbitrary $x\in X^\bullet(T)$, we have $x\leq x^*$ and $x\leq T(x)\leq T(x^*)$, i.e., $T(x^*)\in  \tilde{Z}^\bullet(T)$. As $x^*$ was the minimal element of $\tilde{Z}^\bullet(T)$, it follows that $x^*\leq T(x^*)$, yielding $T(x^*)=x^*$.  However,  $\mathbf{M}(T)$ is the maximal fixed point of $T$ on $[\underline{y}, \overline{y}]$, so that
\begin{equation*}
    x^*\leq \mathbf{M}(T),
\end{equation*}
{thus, $\mathbf{M}(T)= x^*$}, which completes the proof.
\end{proof}

\section{Monotone perturbations}\label{monotonic}

We now prove a series of lemmas that are instrumental in establishing Theorem \ref{thm:StabilityMinMax} in the subsequent section. The latter is a form of stability result for perturbations of the operators $\mathbf{m}$ and $\mathbf{M}$. More specifically, it turns out that the minimal and maximal solutions of the map $y\mapsto S(f,\Phi(y))$ are stable in the norm of $H$ with respect to perturbations in $L^\infty(\Omega)\hookrightarrow V'$ of the forcing term (under certain assumptions on $\Phi$), i.e., if $\{f_n\}$ is in $L_\nu^\infty(\Omega)$ and $f_n\rightarrow f^*$ in $L^\infty(\Omega)$, then 
\begin{equation*}
\mathbf{m}(f_n)\rightarrow\mathbf{m}(f^*)\quad \text{ and }\quad \mathbf{M}(f_n)\rightarrow\mathbf{M}(f^*)\quad \text{ in } H.
\end{equation*}
The strategy of the proof consists in considering the cases of increasing and decreasing sequences of $\{f_n\}$ separately and then {combining} both cases to obtain the final result. This approach is due to the different nature of these cases as indicated in Theorem \ref{thm:approximationMinMax} and Example \ref{ex:Counterexample}. It can also be corroborated by the different structural hypotheses of Lemma \ref{lemma:Decreasing}, \ref{lemma:Increasing}, \ref{lemma:DecreasingM} and \ref{lemma:IncreasingM}. As expected, stability results associated to one-sided perturbations are more amenable than general ones.

In this section, all sequences of forcing terms $\{f_n\}$ are assumed to satisfy $0\leq f_n\leq F$ for all $n\in \mathbb{N}$ and some $F\in V'$ such $F\geq 0$. Further, we consider the interval $[\underline{y}, \overline{y}]$, with $\underline{y}=0$, and $\overline{y}\in V$ such that
\begin{equation}
\langle A(\overline{y}),v\rangle=\langle F,v\rangle, \quad \forall v\in V.
\end{equation}
For any $f$ with $0\leq f\leq F$, we observe {by \eqref{eq:subsol} and \eqref{eq:supsol}} that 
\begin{equation*}
0\leq S(f,\Phi(0)) \qquad\text{and}\qquad S(f,\Phi(\overline{y}))\leq  S(F,+\infty)=\overline{y}.
\end{equation*}
Hence, we denote by $\mathbf{m}(f)$ and $\mathbf{M}(f)$  the minimal and maximal fixed points of the map $y\mapsto S(f, \Phi(y))=S(f,y)$, respectively, on the interval $[\underline{y}, \overline{y}]=[0, A^{-1}(F)]$. Note that $\mathbf{m}(f)$ and $\mathbf{M}(f)$  are well defined according to Theorem \ref{thm:Tartar}.  

In the following lemma, we start by considering the behavior of $\{\mathbf{m}(f_n)\}$ for non-increasing sequences $\{f_n\}$.

\begin{lem}[\textsc{Non-increasing Sequences of $\mathbf{m}$}]\label{lemma:Decreasing}
Suppose that the following hold true:
\begin{itemize}
  \item[(i)] The sequence $\{f_n\}$ in $L_\nu^\infty(\Omega)$ is non-increasing and $\lim_{n\rightarrow\infty}f_n= f^*$ in $L^\infty(\Omega)$ for some  $f^*\in L_\nu^\infty(\Omega)$.
  \item[(ii)] The upper bound mapping $\Phi$ satisfies  
  \begin{equation*}
\lambda  \Phi(y)\geq \Phi(\lambda  y), \quad\text{ for all } \quad \lambda> 1,\: y\in  H^+,
\end{equation*}
and if $\{v_n\}$ is bounded in $V$ and $v_n\downarrow v$ in $H$, then $\Phi(v_n)\rightarrow \Phi(v)$ in $H$.
\end{itemize}
Then, it follows that
\begin{equation}\label{eq:ConvergenceDecreasings}
    \mathbf{m}(f_n)\downarrow \mathbf{m}(f^*) \text { in } H, \quad \text{and} \quad \mathbf{m}(f_n)\rightarrow \mathbf{m}(f^*) \text { in } V.
\end{equation}
\end{lem}

\begin{proof} 
The proof is split into several steps for the sake of clarity.

Step 1: We start by showing: \textit{If a sequence $\{z_n\}$ satisfies $z_n\rightharpoonup z^*$ in $V$, for some $z^*$, and is non-increasing and non-negative: $z_n\geq z_{n+1}\geq 0$ for all $n\in \mathds{N}$, then it holds that $S(f_{n}, \Phi(z_{n}))\rightarrow S(f^*, \Phi(z^*))$ in $V$}. We follow closely the ideas in \cite{Toyoizumi1991} and include the proof here for the sake of completeness. Note first that the non-increasing nature of the sequence implies also that $z_{n}\downarrow z^*$ in $H$.


By our hypothesis on $\Phi$, we have that $\Phi(z_{n})\rightarrow \Phi(z^*)$ in $H$ and $\Phi(z_{n})\geq \Phi(z_{n+1})$, which implies that $\mathbf{K}(\Phi(z_{n}))\supset \mathbf{K}(\Phi(z_{n+1}))$ and $\mathbf{K}(\Phi(z_{n}))\supset \mathbf{K}(\Phi(z^*))$. 


Since $\mathbf{K}(\Phi(z^*))$ is non-empty (note that $z^*\geq 0$ and $\Phi(z^*)\geq 0$), we have $0\in \mathbf{K}(\Phi(z^*))$ and  $0\in \mathbf{K}(\Phi(z_{n}))$ for all $n\in \mathds{N}$. Let $w_n:=S(f
_{n}, \Phi(z_{n}))$ and note that by Proposition \ref{Prop:IncreasingSolMap} the associated sequence is decreasing and bounded from below (in the ordering), so that $w_n\downarrow w^*$ in $H$ for some $w^*\in H$.

By definition $\langle A(w_n)-f_{n},v-w_n\rangle\geq 0$ for all $v\in \mathbf{K}(\Phi(z_{n}))$, and then, using the uniform monotonicity of $A$, we have
\begin{align}\label{eq:bound}
    c\|w_n\|^2_{V}\leq \langle A
    (w_n),w_n\rangle\leq \langle f_{n},w_n\rangle .
\end{align}
Also, $\langle f_{n},w_n\rangle\leq(\|f_{n}\|_{V'}) \|w_n\|_{V}$ and $\|f_{n}\|_{V'}\leq C\|f_{n}\|_{L^\infty(\Omega)}<\infty$. Therefore, $\{w_n\}$ is bounded in $V$, and hence for some subsequence  $w_{n_k}\rightharpoonup w^*$ in $V$ (for the same $w^*$ as before). But since $w_n\downarrow w^*$ in $H$, we get $w_{n}\rightharpoonup w^*$ in $V$.

Since $\Phi(z_{n})\rightarrow \Phi(z^*)$  in $H$ and $w_{n}\leq \Phi(z_{n})$, we conclude
\begin{equation*}
    w^*\leq \Phi(z^*), \quad\text{ i.e., } \quad w^*\in \mathbf{K}(\Phi(z^*)).
\end{equation*}

By Minty's Lemma (see {\cite[Lemma 4.2, Section 4:4]{Rodrigues1987}}) applied to the VI arising from $w_{n}=S(f_{n}, \Phi(z_{n}))$, we obtain
\begin{equation*}
    \langle A(v)-f_{n},v-w_{n}\rangle\geq 0, \quad \forall v\in  \mathbf{K}(\Phi(z_{n})),
\end{equation*}
and in particular for all $v\in \mathbf{K}(\Phi(z^*))\subset \mathbf{K}(\Phi(z_{n}))$. As $f_{n}\rightarrow f^*$ in $L^\infty(\Omega)$ (and hence in $V'$), $(v-w_{n})\rightharpoonup (v-w^*)$ in $V$ and $(v-w_{n})\rightarrow (v-w^*)$ in $H$, we have
\begin{equation*}
    \lim_{k\rightarrow\infty}\langle A(v)-f_{n},v-w_{n}\rangle =\langle A(v)-f^*,v-w^*\rangle\geq 0, \quad \forall v\in \mathbf{K}(\Phi(z^*)).
\end{equation*}
Additionally, since $w^*\in \mathbf{K}(\Phi(z^*))$, Minty's Lemma implies
\begin{equation*}
    \langle A(w^*)-f^*,v-w^*\rangle \geq 0, \quad \forall v\in \mathbf{K}(\Phi(z^*)),
\end{equation*}
i.e., $w^*=S(f^*, \Phi(z^*))$.

Given that $w_{n}\rightharpoonup w^*$ in $V$, {$w_{n}\to  w^*$ in $H$}  and $\|f_{n}\|_{{H}}\leq C  \|f_{n}\|_{L^\infty(\Omega)}<\infty$, by 
\begin{equation*}
c\|w_n-w^*\|^2_V\leq \langle A(w_n)-A(w^*) ,w_n-w^*\rangle\leq \langle {-}f_n+A(w^*) ,w^*-w_n\rangle,
\end{equation*}
we have $w_{n}\rightarrow w^*$ in $V$. That is,
\begin{equation}\label{eq:StrongConvSolMap}
     S(f_{n}, \Phi(z_{n}))\rightarrow S(f^*, \Phi(z^*)) \text { in } V.
\end{equation}

Before we continue with the next step of the proof, we define for $f\in V'$ the set-valued mappings
\begin{align*}
    &X(f)=\{x\in H: \underline{y}\leq x\leq \overline{y} \text{ and } x\leq S(f, \Phi(x))\}, \\
        &Y^\bullet(f)=\{x\in H: \underline{y}\leq x  \text{ and } x\geq S(f, \Phi(x))\},\\
            &Z^\bullet(f)=\{x\in X(f): x\leq y \text { for all } y\in Y^\bullet(f)\}.
\end{align*}

Step 2: \textit{Let $\{z_n\}$ be the sequence of Step 1, i.e., $z_n\rightharpoonup z^*$ in $V$ that is also non-increasing in the sense $z_n\geq z_{n+1}\geq 0$ for all $n$. If $z_n\in Z^\bullet(f_n)$, then  $z^*\in Z^\bullet(f^*)$.}

Since $f_n\in L_\nu^\infty(\Omega)$ with $f_n\geq f_{n+1}$ for all $n\in\mathbb{N}$ and $\lim_{n\rightarrow\infty}f_n= f^*\geq \nu>0$ in $L^\infty(\Omega)$, we have that $f^*\leq f_n$ for all $n\in \mathbb{N}$. Hence, $S(f^*, \Phi(x))\leq S(f_n, \Phi(x))$ for all $n\in {\mathbb{N}}$ and we obtain the  inequalities
\begin{equation}\label{eq:Inclusions}
    X(f^*)\subset X(f_n) \text{ and } Y^\bullet(f_n)\subset Y^\bullet(f^*), \text { and hence } Z^\bullet(f^*)\subset Z^\bullet(f_n).
\end{equation}

Let $z_{n}\in Z^\bullet(f_{n})$, then $z_{n}\in X(f_{n})$, i.e., $\underline{y}\leq z_{n}\leq\overline{y}$ and  $z_{n}\leq S(f_{n}, \Phi(z_{n}))$. Therefore, we have { by Step 1} 
\begin{equation}\label{eq:zisinZ}
   \underline{y}\leq z^*\leq\overline{y} \:\text{ and }\: z^*\leq S(f^*, \Phi(z^*)), \text{ and hence } z^*\in X(f^*).
\end{equation}

Let $y\in Y^\bullet(f^*)$ be arbitrary and consider $y_n:=\lambda_n y$, with  $\lambda_n:=\|f_{n}/f\|_{L^\infty(\Omega)}$. Since $\lambda_n\downarrow 1$ (recall $f_n\rightarrow f$ in $L^\infty(\Omega)$ and $f_n,f\in L_\nu^\infty(\Omega)$ for all $n\in\mathbb{N}$), we infer  $\underline{y}\leq\lambda_n\underline{y}\leq\lambda_ny = y_n$. Also, $\lambda_n  f\geq  f_{n}$ and by the structural assumption over $\Phi$, we have $\lambda_n  \Phi(y)\geq \Phi(\lambda_n  y)$.  Furthermore, we obtain the following chain of inequalities
\begin{align*}
    \lambda_n  y&\geq \lambda_n S(f, \Phi( y))= S( \lambda_n f, \lambda_n\Phi( y))\geq  S(f_n, \Phi(\lambda_n   y)),
\end{align*}
{where we have used that $A$ is homogenous of order one. Thus,} $y_n\in  Y^\bullet(f_n)$ and $y_n\rightarrow y$ in $L^\infty(\Omega)$.

Now, we have that $z_n\in Z^\bullet (f_n)$ and $z_n\rightarrow z^*$ in $H$ and $z^*\in X(f^*)$, and for each $n\in \mathds{N}$ we have $z_n\leq \tilde{y}$ for all $\tilde{y}\in Y^\bullet(f_n)$. Choosing $\tilde{y}=\lambda_n  y$ as in the previous paragraph with $y\in Y^\bullet (f^*)$ arbitrary, we have that $z_n\leq\lambda_n y$. Henceforth, $z^*\leq y$ for all $y\in Y^\bullet(f^*)$, i.e., $z^*\in Z^\bullet(f^*)$.


Step 3.  The minimal solutions $\mathbf{m}(f_n)$ and $\mathbf{m}(f^*)$ are well defined as the maximal elements of $Z^\bullet(f_n)$ and $Z^\bullet(f^*)$, respectively. It follows immediately from \eqref{eq:Inclusions} that $\mathbf{m}(f^*)\leq \mathbf{m}(f_n)$, and by the same argument used to derive  \eqref{eq:Inclusions}, we have that $0\leq \mathbf{m}(f_{n+1})\leq \mathbf{m}(f_n)$. Denote $z_n=\mathbf{m}(f_n)$,  since $z_n=S(f_n, \Phi(z_n))$ and $0\in \mathbf{K}(\Phi(z_n))$, a standard monotonicity argument gives $\|z_n\|_{V}\leq M<\infty$. Hence $z_n$ is bounded in $V$, non-increasing and bounded below in order, and $z_n\in Z^\bullet(f_n)$. The monotone behaviour in addition to the boundedness implies that $z_n \rightharpoonup z^*$ in $V$, by Step 2 we have that $z^*\in Z^\bullet(f^*)$. Since $z_n=S(f_{n}, \Phi(z_n))$, by Step 1, we have that $z_n\rightarrow z^*$ in $V$, $z^*=S(f^*, \Phi(z^*))$, i.e., $z^*$ is a fixed point of the map $z\mapsto S(f^*, \Phi(z))$ and hence $\mathbf{m}(f^*)\leq z^*$. From the definition of $Z^\bullet(f^*)$ we infer that  $z^*\leq y$ for all $y\in Y^\bullet(f^*)$, and we readily observe $\mathbf{m}(f^*)\in Y^\bullet(f^*)$, so that $z^*\leq \mathbf{m}(f^*)$, i.e., $\mathbf{m}(f^*)= z^*$.
%
\end{proof}

A fundamental step in the previous lemma utilizes that 
\begin{equation*}
S(f_{n}, \Phi(z_{n}))\rightarrow S(f^*, \Phi(z^*)), \quad \text{in} \quad V
\end{equation*}
when $f_n\to {f^*}$ in $V'$. A sufficient condition for this to hold true is related to the Mosco convergence (see \cite{Rodrigues1987}) of $\{\mathbf{K}(\Phi(z_n))\}$ towards $\mathbf{K}(\Phi(z^*))$:

\begin{definition}[\textsc{Mosco convergence}]\label{definition:MoscoConvergence}
Let $\mathbf{K}$ and $\mathbf{K}_n$, for each $n\in\mathbb{N}$, be non-empty, closed and convex subsets of $V$. {Then} the sequence \textit{$\{\mathbf{K}_n\}$ {is said to} converge to $\mathbf{K}$ in the sense of Mosco} as $n\rightarrow\infty$, {denoted} by {$\mathbf{K}_n\xrightarrow{\:\mathrm{M}\:}\mathbf{K},$} if the following two conditions {are fulfilled}:
\begin{itemize}
  \item[(i)] For each $w\in \mathbf{K}$, there exists $\{w_{n'}\}$ such that $w_{n'}\in \mathbf{K}_{n'}$ for $n'\in \mathbb{N}'\subset \mathbb{N}$ and $w_{n'}\rightarrow w$ in $V$.
  \item[(ii)] If $w_n\in \mathbf{K}_n$ and $w_n\rightharpoonup w$ in $V$ along a subsequence, then $w\in \mathbf{K}$.
\end{itemize}
\end{definition}

Mosco convergence of unilaterally constrained sets is equivalent (in the case when the obstacles are quasi-continuous and $V$ a certain Sobolev space) to convergence of the obstacles in the sense of the capacity (which might be cumbersome to prove beyond rather simple examples). It is also well-known that the convergence of the obstacles in the sense of $L^\infty(\Omega)$ is a sufficient condition for Mosco convergence (in most applications), although this might be rather a strong assumption in some cases. In the previous case, we are able to avoid that strong assumption rather elegantly by assuming only the $H$ convergence of the obstacles. In the next case, for non-increasing sequences, the  $L^\infty(\Omega)$ or $V$ convergence {can} be avoided by using a correction of an argument of Toyoizumi (see \cite{Toyoizumi1991}) by using geometrical considerations of the obstacles. For this matter, we consider the following assumption.

\begin{assumption}\label{PhiAss}
If $v_n\rightharpoonup v$ in $V$, then $\Phi$ satisfies one of the following:
\begin{itemize}
  \item[(a)] $\Phi(v_n)\rightarrow \Phi(v)$ in $L_\nu^\infty(\Omega)$, or $\Phi(v_n)\rightarrow \Phi(v)$ in $V$.
 \item[(b)] $\Phi(v_n)\rightarrow \Phi(v)$ in $H$ and if $v\in V\cap H^+$, then $\Phi(v)\in V$ and $\mathcal{Q} \Phi(v)\geq 0$ in $V$, for some strongly monotone $\mathcal{Q}\in \mathscr{L}(V,V')$, such that $\langle \mathcal{Q}v^-, v^+\rangle\leq 0$ for all $v\in V$.
\end{itemize}  
\end{assumption}

With the above definition in mind we are now in the position to provide the stability result for minimal solutions and for non-decreasing sequences of forcing terms.

\begin{lem}[\textsc{Non-decreasing Sequences for $\mathbf{m}$}]\label{lemma:Increasing}
Suppose the following: 
\begin{itemize}
  \item[(i)] The sequence $\{f_n\}$ in $H^+$ is non-decreasing and $\lim_{n\rightarrow\infty}f_n= f^*$ in $H$ for some $f^*\in H$.
  \item[(ii)] The upper bound mapping $\Phi$ satisfies Assumption \ref{PhiAss}.
\end{itemize}
Then, the following hold true:
\begin{equation*}
    \mathbf{m}(f_n)\uparrow \mathbf{m}(f^*)\text { in } H, \quad \text{and} \quad \mathbf{m}(f_n)\rightarrow \mathbf{m}(f^*) \text { in } V.
\end{equation*}
\end{lem}


\begin{proof}
We use the result of Theorem \ref{thm:approximationMinMax} with $R_n(v):=S(f_n, \Phi(v))$ and $S(v):=S(f^*, \Phi(v))$. The classical continuity result for $f\mapsto S(f,\Phi(y))$ (see \cite{Kinderlehrer}) states:
\begin{equation}\label{eq:ContinuityForcingTerm}
    \|S(f^*, \Phi(y))-S(f_n, \Phi(y))\|_{V}\leq\frac{1}{c}\|f^*-f_n\|_{V'}.
\end{equation}
Since $f_n\rightarrow f^*$ in $V'$ as $n\rightarrow\infty$, we have $S(f_n, \Phi(y))\to S(f^*, \Phi(y))$ in $V$, uniformly on bounded sets for $y$. Additionally, by the usual monotonicity argument and using $v=0$ as a test function, we obtain that $\|S(f, \Phi(y))\|_{V}\leq\frac{1}{c}\|f\|_{V'}$ which implies that the set of fixed points of the maps $y\mapsto S(f_n,\Phi(y))$, for $n\in\mathbb{N}$, and  $y\mapsto S(f^*,\Phi(y))$ is uniformly bounded. Since $S(f_n,\Phi(y))\leq S(f_{n+1},\Phi(y))\leq S(f^*,\Phi(y))$ we are only left to prove that 
\begin{equation}\label{eq:ContinuityObstacle2}
    \lim_{n\rightarrow\infty}S(f_n, \Phi(v_n))=S(f^*, \Phi(v)) \text{ in } V,
\end{equation}
{for any sequence $\{v_n\}$ in $V$ satisfying $v_n\leq v_{n+1}$ for all $n$, and $v_n\rightharpoonup v$ in $V$}. This will be achieved by proving Mosco convergence of the associated constraints. 
Now we consider the two possible cases for $\Phi$ based on Assumption \ref{PhiAss}:
\begin{itemize}
\item[\textit{(a)}]  Since $\{v_n\}$ in $V$ satisfies $v_n\rightharpoonup v$ in $V$ and $\Phi(v_n)\rightarrow \Phi(v)$ in $L_\nu^\infty(\Omega)$, it follows that  $\mathbf{K}(\Phi(v_n))\rightarrow \mathbf{K}(\Phi(v))$ in the sense of Mosco {by a direct scaling argument} (see for example {\cite[Proposition 6.6, Section 4:7]{Rodrigues1987}} for a further  general result).

Suppose that $\Phi(v_n)\rightarrow \Phi(v)$ in $V$. Let $w\leq \Phi(v)$, and consider $w_n=w-\Phi(v)+\Phi(v_n)$. Then, $w_n\leq\Phi(v_n)$ and also $w_n\to w$ in $V$, i.e., \emph{(i)} in Definition \ref{definition:MoscoConvergence} holds. Furthermore, if $y_n\leq \Phi(v_n)$ and $y_n\rightharpoonup y$ in $V$, then by Mazur's lemma it follows that $y\leq \Phi(v)$ which proves \emph{(ii)} in Definition \ref{definition:MoscoConvergence}.

\item[\textit{(b)}] Recall  $\{v_n\}$ in $V$ satisfies $ v_n\leq v_{n+1}$ for all $n$, and $v_n\rightharpoonup v$ in $V$. Then{, we observe} $\Phi(v_n)\leq \Phi(v_{n+1})$, {given that} $\Phi$ is increasing, and   $\Phi(v_n)\rightarrow \Phi(v)$ in $H$ by {initial assumption}.  Hence, if $y_n\leq \Phi(v_n)$ and $y_n\rightharpoonup y$ in $V$, then by Mazur's lemma it follows that $y\leq \Phi(v)$, which proves \textit{(ii)} in Definition \ref{definition:MoscoConvergence}. 

In order to prove \textit{(i)} in Definition \ref{definition:MoscoConvergence}, we now follow a modification of the argument in \cite{Toyoizumi1991}. Let $w\in V$ such that $w\leq \Phi(v)$ and $w_n$ be defined by
\begin{equation}\label{eq:SingPert}
\left\langle r_n\mathcal{Q}w_n+w_n, v\right\rangle = ( \phi_n, v),\text{ for all } v\in V,
\end{equation}
where $r_n:=\|\phi_n- w\|_H$ and $\phi_n:=\min(w, \Phi(v_n))$, and note that $\phi_n\rightarrow w$ in $H$ and $w\in V$. Then, we can prove that $w_n\to w$ in $V$. Since  $\mathcal{Q}$ is linear, bounded, and $\langle \mathcal{Q}v, v\rangle\geq c\|v\|_V^2$ for all $v\in V$, from the definition of $w_n$ we have
\begin{align}\notag
	r_nc\|w_n-w\|_V^2+\|w_n-w\|_H^2&\leq \left\langle (r_n\mathcal{Q}+I)(w_n-w), w_n-w\right\rangle\\\label{eq:usesing}
	&\leq \langle \phi_n-w, w_n-w\rangle- r_n\langle\mathcal{Q}w, w_n-w\rangle\\\notag
	&\leq  r_n(C_p+\|\mathcal{Q}w\|_{V'})\| w_n-w\|_V,
\end{align}
where $C_p$ is the constant for the embedding $V \hookrightarrow H$, {and recall that $r_n=\|\phi_n- w\|_H$}. This implies that, $\{w_n\}$ is bounded in $V$, so that $w_n\rightharpoonup w^*$ (along a subsequence) for some $w^*\in V$. By taking the limit in \eqref{eq:SingPert}, it is shown that $w^*=w$ and that $w_n\rightharpoonup w^*$ in $V$ not only along a subsequence. It further follows that $w_n\to w$ in $H$, and since from \eqref{eq:usesing} we observe
\begin{align}
	r_nc\|w_n-w\|_V^2+\|w_n-w\|_H^2
	&\leq  r_n(\|w_n-w\|_H+\langle\mathcal{Q}w, w-w_n\rangle),
\end{align}
we have that $w_n\to w$ in $V$.

Now we prove that $w_n\leq \Phi(v_n)$. Consider $v=(w_n-\Phi(v_n))^+$ and subtract $\left\langle r_n\mathcal{Q}\Phi(v_n)+\Phi(v_n), v\right\rangle$ from both sides of \eqref{eq:SingPert}. Then, we get
\begin{align*}
 & r_n\left\langle\mathcal{Q}(w_n-\Phi(v_n)), (w_n-\Phi(v_n))^+\right\rangle +\|(w_n-\Phi(v_n))^+\|^2_H=\\
 &\quad -r_n\left\langle \mathcal{Q}\Phi(v_n), (w_n-\Phi(v_n))^+\right\rangle+( \min(w, \Phi(v_n))-\Phi(v_n), (w_n-\Phi(v_n))^+).
\end{align*}
Note that $\min(w, \Phi(v_n))-\Phi(v_n)\leq 0$ and by assumption $\mathcal{Q}\Phi(v_n)\geq 0$. Therefore the right hand side is less or equal to zero. Additionally, since $\mathcal{Q}$ is linear, $\langle \mathcal{Q}v^-, v^+\rangle\leq 0$, and $\langle \mathcal{Q}v, v\rangle\geq c\|v\|_V^2$ for all $v\in V$, {we observe that  
\begin{equation*}
c\|v^+\|_V^2\leq\langle \mathcal{Q}v^+, v^+\rangle \leq\langle \mathcal{Q}v^+, v^+\rangle -\langle\mathcal{Q}v^-, v^+\rangle=\langle\mathcal{Q}v, v^+\rangle.	
\end{equation*}
Thus}
\begin{align*}
&r_n c\|(w_n-\Phi(v_n))^+\|_V^2 +\|(w_n-\Phi(v_n))^+\|^2_H\leq\\
&\quad r_n\left\langle\mathcal{Q}(w_n-\Phi(v_n))^+, (w_n-\Phi(v_n))^+\right\rangle +\|(w_n-\Phi(v_n))^+\|^2_H\leq 0.
\end{align*}
This yields $w_n\leq\Phi(v_n)$, i.e., \textit{(i)} in Definition \ref{definition:MoscoConvergence} holds
\end{itemize}
\end{proof}

Lemma \ref{lemma:Decreasing} and \ref{lemma:Increasing} are associated to non-increasing and non-decreasing sequences of minimal solutions. In the following we establish Lemma \ref{lemma:DecreasingM} and \ref{lemma:IncreasingM} that deal with the analogous results but for maximal solutions.

\begin{lem}[\textsc{Non-increasing Sequences for $\mathbf{M}$}]\label{lemma:DecreasingM}
Suppose the following: 
\begin{itemize}
  \item[(i)] The sequence $\{f_n\}$ in $H^+$ is non-increasing and $\lim_{n\rightarrow\infty}f_n= f^*$ in $H$ for some $f^* \in H$.
  \item[(ii)] The upper bound mapping $\Phi$ satisfies that  if $\{v_n\}$ is bounded in $V$, $v_n\downarrow v$ in $H$, then $\Phi(v_n)\rightarrow \Phi(v)$ in $H$.
\end{itemize}
Then, we have
\begin{equation}\label{eq:ConvergenceDecreasingsM}
    \mathbf{M}(f_n)\downarrow \mathbf{M}(f^*) \text{ in } H, \quad\text{and} \quad \mathbf{M}(f_n)\rightarrow \mathbf{M}(f^*) \text { in } V.
\end{equation}
\end{lem}

\begin{proof}As obtained in the proof of Lemma \ref{lemma:Increasing}, we have that $S(f_n, \Phi(y))\to S(f^*, \Phi(y))$ in $V$ and that the set of fixed points of the maps $y\mapsto S(f_n,\Phi(y))$, for $n\in\mathbb{N}$, and  $y\mapsto S(f^*,\Phi(y))$ are uniformly bounded in $V$. 

Let $\{v_n\}$ be such that $v_n\rightharpoonup v$ in $V$ and $v_n\geq v_{n+1}\geq 0$ for all $n$, then $v_n\rightarrow v$ in $H$, $\Phi(v_n)\geq \Phi(v_{n+1})\geq 0$ and $\Phi(v_n)\rightarrow \Phi(v)$ in $H$. Note that $f_n\rightarrow f^*$ in $H$ is enough for step 1 of the proof of Lemma \ref{lemma:Decreasing} to hold, i.e., 
\begin{equation*}
\lim_{n\rightarrow\infty}S(f_n, \Phi(v_n))=S(f, \Phi(v)), \quad \text{ in } V.
\end{equation*}

Therefore applying Theorem \ref{thm:approximationMinMax} to $T_n(v):=S(f_n, \Phi(v))$ and $S(v):=S(f^*, \Phi(v))$, we obtain that \eqref{eq:ConvergenceDecreasingsM} holds true.
\end{proof}

\begin{lem}[\textsc{Non-decreasing Sequences for $\mathbf{M}$}]\label{lemma:IncreasingM}
Suppose the following: 
\begin{itemize}
  \item[(i)] The sequence $\{f_n\}$ in $L_\nu^\infty(\Omega)$ is non-decreasing and  $\lim_{n\rightarrow\infty}f_n= f^*$ in $L^\infty(\Omega)$ for some $f^*$.
  \item[(ii)] The upper bound mapping $\Phi$ satisfies   
  \begin{equation*}
\lambda  \Phi(y)\leq \Phi(\lambda  y), \quad\text{ for all }\quad  0<\lambda< 1,\: y\in  H^+,
\end{equation*}
and Assumption \ref{PhiAss}.
\end{itemize}
Then, we have
\begin{equation}\label{eq:ConvergenceIncreasingsM}
    \mathbf{M}(f_n)\uparrow \mathbf{M}(f^*) \text{ in } H, \quad\text{ and } \quad \mathbf{M}(f_n)\rightarrow \mathbf{M}(f^*) \text { in } V.
\end{equation}
\end{lem}

\begin{proof}
For $f$ define the set-valued mappings
\begin{align*}
    &X^\bullet(f)=\{x\in H:   x\leq \overline{y} \text{ and } x\leq S(f, \Phi(x))\}, \\
        &Y(f)=\{x\in H: \underline{y}\leq x \leq \overline{y}  \text{ and } x\geq S(f, \Phi(x))\},\\
            &\tilde{Z}^\bullet(f)=\{y\in Y(f): x\leq y \text { for all } x\in X^\bullet(f)\}.
\end{align*}
If $\{z_n\}$ satisfies $z_n\rightharpoonup z^*$ in $V$ for some $z^*\in V$ and is non-decreasing, i.e., $z_{n}\leq z_{n+1}$ for all $n\in \mathds{N}$, then 
\begin{equation}\label{eq:ConvSolMapping}
S(f_{n}, \Phi(z_{n}))\rightarrow S(f^*, \Phi(z^*)), \quad \text{ in } V,
\end{equation}
as proven in Lemma  \ref{lemma:Increasing}. We now show that if $z_n\in \tilde{Z}^\bullet(f_n)$, for all $n\in\mathbb{N}$, then $z^*\in \tilde{Z}^\bullet(f^*)$.

Since $f_n\leq f_{n+1}$, for all $n\in\mathbb{N}$, and $\lim_{n\rightarrow\infty}f_n= f^*$ in $L^\infty(\Omega)$, we have that $f_n\leq f^*$ and $S(f_n,\Phi(x))\leq S(f^*,\Phi(x))$, for all $n\in\mathbb{N}$. Therefore,
\begin{equation}\label{eq:InclusionsM}
    X^\bullet(f_n)\subset X^\bullet(f^*)  \text{ and } Y(f^*) \subset Y(f_n), \text { and hence } \tilde{Z}^\bullet(f^*)\subset \tilde{Z}^\bullet(f_n).
\end{equation}
Also, $z_{n}\in \tilde{Z}^\bullet(f_{n})$ and hence $z_{n}\in Y(f_{n})$, i.e., $\underline{y}\leq z_{n}\leq\overline{y}$ and  $z_{n}\geq S(f_{n}, \Phi(z_{n}))$. Therefore, by \eqref{eq:ConvSolMapping} and since $z_n\rightarrow z^*$ in $H$ (note that $z_{n}\leq z_{n+1}\leq\overline{y}$) we observe that
\begin{equation}\label{eq:zisinZM}
   \underline{y}\leq z^*\leq\overline{y} \:\text{ and }\: z^*\geq S(f^*, \Phi(z^*)), \text{ and hence }z^*\in Y(f^*).
\end{equation}

Let $x\in X^\bullet(f^*)$ be arbitrary and consider $x_n:=\lambda_n x$, with  $\lambda_n:=\operatorname{ess\:inf}|f_n/f^*|$ which yields $\lambda_n\uparrow 1$. Indeed, since $f_n\leq f_{n+1}$, we have  $|f_n/f^*|\leq |f_{n+1}/f^*|\leq 1$ and 
\begin{equation*}
\left|1-\operatorname{ess\:inf}\left|\frac{f_n}{f^*}\right|\right|={\operatorname{ess\:sup}}\left|1-\left|\frac{f_n}{f^*}\right|\right|\leq \frac{\|f^*-f_n\|_{L^\infty(\Omega)}}{\nu},
\end{equation*}
where we have used that $f^*\in L_\nu^\infty(\Omega)$, and the result follows from the assumed convergence $f_n\rightarrow {f^*}$ in $L^\infty(\Omega)$. 

Therefore,  $x_n =\lambda_n x \leq\lambda_n\overline{y} \leq \overline{y}$, $\lambda_n  {f^*}\leq  f_{n}$, and by the structural assumption on $\Phi$, we have $\lambda_n  \Phi(y)\leq \Phi(\lambda_n  y)$.  Furthermore, we obtain the following chain of inequalities:
\begin{align*}
    \lambda_n  x&\leq \lambda_n S({f^*}, \Phi(x))= S( \lambda_n {f^*}, \lambda_n\Phi( x))\leq  S(f_n, \Phi(\lambda_n   x)),
\end{align*}
{where we have used that $A$ is homogenous of order one. Therefore,} $x_n\in  X^\bullet(f_n)$ and $x_n\rightarrow x$ in $H$.

Since, {by hypothesis}, $z_{n}\in \tilde{Z}^\bullet(f_{n})$,  we have  $x_n\leq z_{n}$ given the fact that $x_n\in  X^\bullet(f_n)$. Additionally, along a subsequence we have that $x_n\rightarrow x$ and $z_n\rightarrow z^*$ in $H$ so that $x\leq z^*$. However, $x\in X^\bullet(f^*)$ was arbitrary and hence, by \eqref{eq:zisinZM}, $z^*\in \tilde{Z}^\bullet(f^*)$.

Since $\mathbf{M}(f_n)$ and $\mathbf{M}(f^*)$ are well-defined as the minimal elements of $\tilde{Z}^\bullet(f_n)$ and $\tilde{Z}^\bullet(f^*)$, respectively, it follows immediately from \eqref{eq:InclusionsM} that $\mathbf{M}(f_n)\leq \mathbf{M}(f^*)$, and furthermore, we have that $\mathbf{M}(f_{n})\leq \mathbf{M}(f_{n+1})$. Denoting $z_n=\mathbf{M}(f_n)$, we have $z_n=S(f_n, \Phi(z_n))$, and since $0\in \mathbf{K}(\Phi(z_n))$, a {strong monotonicity argument} gives $\|z_n\|_{V}\leq\frac{1}{c}\|f_n\|_{V'}\leq \frac{1}{c}\|F\|_{V'}<\infty$. Hence, $z_n$ is bounded in $V$, non-decreasing in order and $z_n\in \tilde{Z}^\bullet(f_n)$. Therefore, by the above paragraphs, we have that $z_n=\mathbf{M}(f_n)\rightharpoonup z^*$ in $V$ and  $z^*\in \tilde{Z}^\bullet(f^*)$ and additionally, since $z_n=S(f_{n}, \Phi(z_n))$, by \eqref{eq:ConvSolMapping}, we have that $z_n\rightarrow z^*$ in $V$, $z^*=S(f^*, \Phi(z^*))$, i.e., $z^*$ is a fixed point of the map $z\mapsto S(f^*, \Phi(z))$ and hence $z^*\leq \mathbf{M}(f^*)$. By definition of $\tilde{Z}^\bullet(f^*)$, we have that  $x\leq z^*$ for all $x\in X^\bullet(f^*)$ and we readily observe $\mathbf{M}(f^*)\in X^\bullet(f^*)$, so that $ \mathbf{M}(f^*) \leq z^*$, i.e., $\mathbf{M}(f^*)= z^*$.
%
\end{proof}

{
\begin{remark}\label{rem:equiv}
	Note that condition
	  \begin{equation*}
\lambda  \Phi(y)\geq \Phi(\lambda  y), \quad\text{ for all } \quad \lambda> 1,\: y\in  H^+,
\end{equation*}
in Lemma \ref{lemma:Decreasing}, and condition 
  \begin{equation*}
\lambda  \Phi(y)\leq \Phi(\lambda  y), \quad\text{ for all }\quad  0<\lambda< 1,\: y\in  H^+,
\end{equation*}
in Lemma \ref{lemma:IncreasingM} are equivalent: Consider the change of variables $y=\frac{1}{\lambda}\tilde{y}$ with $\tilde{y}\in  H^+$.
\end{remark}
}

\section{Non-monotone perturbations and problem \eqref{eq:CQVI}} \label{control}

We are now in the position to establish our fundamental result concerning the behavior of the maps $f\mapsto \mathbf{m}(f)$ and $f\mapsto \mathbf{M}(f)$. Although the hypotheses of lemmas \ref{lemma:Decreasing}, \ref{lemma:Increasing}, \ref{lemma:DecreasingM} and \ref{lemma:IncreasingM} seem to be quite diverse, when considering the intersection in the following theorem, the assumptions are simplified. As in the previous section we assume that $0\leq f_n\leq F$ for any sequence $\{f_n\}$ and that $[\underline{y},\overline{y}]=[0, A^{-1}F]$.

\begin{thm}\label{thm:StabilityMinMax}
Let $\{f_n\}$ in $L^\infty_\nu(\Omega)$ be such that $\lim f_n=f^*$ in $L^\infty(\Omega)$ for some $f^*$, suppose that the upper bound mapping $\Phi: H^+ \rightarrow  H^+$ satisfies Assumption \ref{PhiAss}, { and that $\lambda  \Phi(y)\geq \Phi(\lambda  y)$ for any $\lambda> 1$ and any $y\in  H^+$} . Then the following hold true:
\begin{itemize}
\item[(i)] { The sequence of minimal solutions satisfy } 
					\begin{equation}\label{eq:ConvergenceMin}
    \mathbf{m}(f_n)\rightarrow \mathbf{m}(f^*) \text { in } H, \qquad \text{and} \qquad  \mathbf{m}(f_n)\rightharpoonup \mathbf{m}(f^*) \text { in } V.
\end{equation}
\item[(ii)] { The sequence of maximal solutions satisfy } 
					\begin{equation}\label{eq:ConvergenceMax}
    \mathbf{M}(f_n)\rightarrow \mathbf{M}(f^*) \text { in } H,  \qquad \text{and} \qquad  \mathbf{M}(f_n)\rightharpoonup \mathbf{M}(f^*) \text { in } V.
\end{equation}
\end{itemize}
\end{thm}
\begin{proof}
Define $\hat{f}_n:=\inf_{m\geq n} f_m$ and $\check{f}_n:=\sup_{m\geq n} f_m$, so that $0\leq\nu\leq \hat{f}_n\leq \hat{f}_{n+1}\leq F$, $F\geq \check{f}_n\geq \check{f}_{n+1}\geq\nu>0$ for all $n\in \mathbb{N}$, and also $\lim_{n\rightarrow \infty} \hat{f}_n=\lim_{n\rightarrow \infty} \check{f}_n=f^*$ in $L^\infty(\Omega)$. Since $0\leq \hat{f}_n\leq f_n\leq \check{f}_n\leq F$ and the map $H^+\ni y\mapsto S(f,\Phi(y))$ is increasing for any $f\in V'$, we have that $\mathbf{m}(\hat{f}_n), \mathbf{m}(f_n)$, $\mathbf{m}(\check{f}_n)$ and $\mathbf{m}(f^*)$ as well as $\mathbf{M}(\hat{f}_n), \mathbf{M}(f_n)$, $\mathbf{M}(\check{f}_n)$ and $\mathbf{M}(f^*)$ are well defined (note that $0\leq f^*\leq F$), respectively. Moreover, we have that 
\begin{equation*}
0\leq S(\hat{f}_n,\Phi(y))\leq S(f_n,\Phi(y))\leq S(\check{f}_n,\Phi(y))\leq \overline{y} \qquad \forall y\in[0, \overline{y}], n\in\mathbb{N}.
\end{equation*}
Hence from the inclusions \eqref{eq:IneqSets1} and \eqref{eq:IneqSets2}, we obtain {from \eqref{RnIneq} that} 
\begin{equation}\label{eq:MinIneq}
0\leq\mathbf{m}(\hat{f}_n)\leq \mathbf{m}(f_n)\leq \mathbf{m}(\check{f}_n)\leq \overline{y}, \qquad \forall n\in\mathbb{N},
\end{equation}
and from the inclusions \eqref{eq:IneqSets3} and \eqref{eq:IneqSets4} that
\begin{equation}\label{eq:MaxIneq}
0\leq\mathbf{M}(\hat{f}_n)\leq \mathbf{M}(f_n)\leq \mathbf{M}(\check{f}_n)\leq \overline{y}, \qquad \forall n\in\mathbb{N}.
\end{equation}

Then, by lemmas \ref{lemma:Decreasing}, \ref{lemma:Increasing}, \ref{lemma:DecreasingM} and \ref{lemma:IncreasingM} we have that $\mathbf{m}(\hat{f}_n)\rightarrow \mathbf{m}(f^*)$, $\mathbf{m}(\check{f}_n)\rightarrow \mathbf{m}(f^*)$, $\mathbf{M}(\hat{f}_n)\rightarrow \mathbf{M}(f^*)$ and $\mathbf{M}(\check{f}_n)\rightarrow \mathbf{M}(f^*)$, all in $V$ and $H$. Hence, we find
\begin{equation*}
\mathbf{m}(f_n)\rightarrow \mathbf{m}(f^*) \text { in } H\qquad \text{and} \qquad  \mathbf{M}(f_n)\to \mathbf{M}(f^*) \text { in } H
\end{equation*}
by \eqref{eq:MinIneq} and \eqref{eq:MaxIneq}. Since $\{\mathbf{m}(f_n)\}$ and $\{\mathbf{M}(f_n)\}$ are bounded in $V$, they are also weakly convergent (along a subsequence) to $\mathbf{m}(f^*)$ and $\mathbf{M}(f^*)$, respectively. However, since the entire sequences $\{\mathbf{m}(f_n)\}$ and $\{\mathbf{M}(f_n)\}$ strongly converge in $H$, it further follows that they converge weakly (not only along a subsequence) in $V$. Hence \eqref{eq:ConvergenceMin} and \eqref{eq:ConvergenceMax} hold true.
\end{proof}

With the aid of the previous theorem we can now formulate the result that proves the well-posedness of  \eqref{eq:CQVIr}.
We assume that   
\begin{equation}\label{eq:UadLoc}
U_{\mathrm{ad}}\subset \{f\in L_\nu^\infty(\Omega):  f\leq F\},
\end{equation}
for some $F\in V'$.  As in previous sections $\underline{y}=0$ and $\overline{y}=A^{-1}F$, so that $\mathbf{m}(f)$ and $\mathbf{M}(f)$ are defined as the minimal and maximal solutions, respectively, of the QVI in \eqref{eq:QVI}. Hence, the reduced version of \eqref{eq:CQVI} is given by
\begin{equation}\label{eq:CQVIr}\tag{$\tilde{\mathbb{P}}$}
\begin{split}
&\text{minimize }  J_1(\mathbf{m}(f),\mathbf{M}(f))+J_2(f),\\
&\text{subject to }  f\in U_{\text{ad}}.
\end{split}
\end{equation}
The well posedness of \eqref{eq:CQVIr} (and hence of \eqref{eq:CQVI}) is now shown in the following result.

\begin{thm}\label{thm:Existence}
Suppose that
\begin{enumerate}
\item[(i)] $J_1:V\times V\to \mathbb{R}$ is weakly lower semicontinuous,
\item[(ii)] $J_2: L^\infty(\Omega)\to \mathbb{R}$ is continuous, \quad { or}
\item[(ii')] { $J_2: U\to \mathbb{R}$ is coercive and  weakly lower semicontinuos} where { $U$ is a reflexive Banach space.}
\end{enumerate}
and both $J_1$ and $J_2$ are bounded from below. In addition suppose that {$U_{\mathrm{ad}}$ satisfies \eqref{eq:UadLoc}, is closed in $L^\infty(\Omega)$ (if \textit{(ii)} holds) or is weakly closed in $U$ (if \textit{(ii')} holds). Further, }for each $\alpha>0$ let the set 
\begin{equation*}
\{f\in U_{\mathrm{ad}}: J_2(f)\leq \alpha \}
\end{equation*}
be sequentially compact in $L^\infty(\Omega)$. {Additionally, assume that $\Phi$ satisfies Assumption \ref{PhiAss}, {and that $\lambda  \Phi(y)\geq \Phi(\lambda  y)$ for any $\lambda> 1$ and any $y\in  H^+$}}. Then, problem \eqref{eq:CQVIr}, and hence problem \eqref{eq:CQVI}, admits a solution.
\end{thm}

\begin{proof}
Given Theorem \ref{thm:StabilityMinMax}, the proof is just an application of the direct method of the calculus of variations.
\end{proof}

\subsection{Applications}

We finally return to the applications considered earlier in the paper.

\subsubsection{QVIs arising by coupling VIs and PDEs} We consider the problem class as described in section \ref{sec:VIsPDEs} and study conditions on $G, B,$ and $L$ to establish stability of minimum and maximum solutions to the QVI of interest. Recall that $\Phi$ in this setting is defined as $\Phi(y)=Lz(y)$ where $z(y)$ solves
\begin{align*}
\langle Bz+G(Lz,y)-g, w\rangle&= 0 &\forall w\in W,
\end{align*}
for $y\in H$.

\begin{proposition} Under the assumptions of section \ref{sec:VIsPDEs} suppose {either that $(a)$ } If $\lambda> 1$ and $v\in H^+$, then for all $z_1,z_2\in V$, it holds true that 
\begin{equation*}
(\lambda G(Lz_2,v)-G(Lz_1,\lambda v),(z_1-\lambda z_2)^+)\leq 0,
\end{equation*}
{or $(b)$} If $\lambda\in (0,1)$ and $v\in H^+$, then for all $z_1,z_2\in V$, it holds true that
\begin{equation*}
(G(Lz_2,\lambda v)-\lambda G(Lz_1,v),(\lambda z_1-z_2)^+)\leq 0.
\end{equation*}
Then, we have $\lambda  \Phi(v)\geq \Phi(\lambda  v)$ for all $\lambda> 1$ and $v\in H^+$.

\end{proposition}

\begin{proof}
Let $w=z(\lambda v)-\lambda z(v)$ for $\lambda\geq 1$ and $v\in H^+$. Since $B$ is coercive and $\langle B w^-, w^+ \rangle \leq 0$ we observe that
\begin{align*}
c|w^+|_W^2\leq\langle Bw, w^+\rangle= (\lambda G(Lz(v),v)-G(Lz(\lambda v),\lambda v) ,w^+)\leq 0,
\end{align*}
i.e., $z(\lambda v)-\lambda z(v)\leq 0$, so that $Lz(\lambda v)-L(\lambda z(v))\leq 0$ given that $L$ preserves order. Hence, it follows that $\lambda  \Phi(v)\geq \Phi(\lambda  v)$.

Similarly, consider $w=\lambda z(v)-z(\lambda v)$ for $0<\lambda< 1$ and $v\in H^+$. Then,
\begin{align*}
c|w^+|_W^2\leq\langle Bw, w^+\rangle= (G(Lz(\lambda v)-\lambda G(Lz(v),v),\lambda v) ,w^+)\leq  0,
\end{align*}
i.e., $\lambda z(v)-z(\lambda v)\leq 0$, so that $\lambda L z(v)- Lz(\lambda v)\leq 0$ and hence $\lambda  \Phi(v)\leq \Phi(\lambda  v)$. {Thus, the result follows by the equivalence shown in Remark \ref{rem:equiv}}.
\end{proof}

Note that the problem given in Example \ref{ex1} satisfies the assumptions of the above proposition. Additionally, if the solution to $By=h$ satisfies $|y|_{H^2(\Omega)}\leq M|h|_{L^2(\Omega)}$ with $M$ independent of $h$, {and $L\in \mathscr{L}(H^2(\Omega))$} then for dimensions $N=1,2,3$ it is direct to infer that $\Phi:{L_+^2(\Omega)}\to L_\nu^\infty(\Omega)$ is completely continuous via Sobolev compact embeddings {since $\Phi(u)\in H^2(\Omega)$}; { see \cite[Rellich-Kondrachov Theorem, section 6.3]{MR2424078}}. {Note that the $H^2$ estimate does not necessarily require a smooth boundary: In fact, for a second order elliptic operator, a convex domain $\Omega$ is enough; see \cite[Theorem 3.2.1.2]{MR775683}}. Hence, all hypotheses of Theorem \ref{thm:StabilityMinMax} are met, and the minimum and maximum solutions are stable for perturbations of $f$ in $L^\infty(\Omega)$. Finally, if
\begin{equation*}
\{f\in U: 0<\nu \leq f\leq  F \:\text{and}\: \|f\|_U\leq \alpha \},
\end{equation*}
is sequentially compact in $L^\infty(\Omega)$ for each $\alpha>0$, we have that Problem \ref{eq:ControlVIPDE} has a solution. Further note that this last compactness assumption is satisfied for Example \ref{ex1}.

\subsubsection{The impulse control problems}

The previous can be directly applied to the impulse control problem in the bounded case. Let $\Omega=(0,1)$. Then, we have that $V=H^1(\Omega)$ compactly embeds into $C(\overline{\Omega})$, and hence it follows that for 
\begin{equation*}
(\Phi y)(x)=k+\mathrm{essinf}_{x+\xi \in \overline{\Omega}} (c_0(\xi)+y(x+\xi)),
\end{equation*}
with $k>0$ and $c_0$ continuous, we have that if $v_n\rightharpoonup v$ in $V$, then $\Phi(v_n)\rightarrow \Phi(v)$ in $C(\overline{\Omega})\subset L^\infty(\Omega)$. Hence, $\Phi$ satisfies Assumption \ref{PhiAss}. Furthermore, it follows that  $\lambda  \Phi(y)\geq \Phi(\lambda  y)$ for any $\lambda> 1$ and any $y\in H^+$.

Consider $U=H^1(\Omega)$ and $U_{\mathrm{ad}}:=\{f\in U: 0<\nu\leq f\leq F\}$ for  some $F\in H^1(\Omega)^*$, $J_1(a,b)=J_1(a)=\int_{\Omega}(s-a(x))^2\mathrm{d}x$ for some $s>0$, and $J_2(f):=\frac{\lambda}{2}|f|^2_{U}$. It follows that $\{f\in U_{\mathrm{ad}}: J_2(f)\leq \alpha \}$ is sequentially compact in $L^\infty(\Omega)$ for each $\alpha>0$ and that problem \eqref{eq:CQVIr} (which is the reduced version of problem \eqref{eq:Opt}) has a solution by Theorem \ref{thm:Existence}. {The application of \S \ref{sec:unboundedCase} can be treated mutatis mutandis.}

\section{Conclusion}

We have developed a theoretical framework for the study of optimal control problems with QVI constraints. Specifically, the reduced optimization problem of interest involves minimal and maximal points of the solution set to the QVI. The existence question reduces to the stability of two operators $\mathbf{m}$ and $\mathbf{M}$, that relate the solution set of the QVI to its minimal and maximal elements, respectively. Stability of such maps was developed for monotonic and non-monotonic perturbations, and we have applied such results to applications involving QVIs arising from impulse control problems and problems involving VIs coupled with nonlinear PDEs.

\bibliography{DatabaseBibliography}{}


\bibliographystyle{abbrv}
\end{document}